\documentclass[a4paper,10pt]{amsart}
\usepackage[utf8]{inputenc}
\usepackage{amssymb}
\usepackage{amsmath}
\usepackage{bbm}
\usepackage{mathrsfs}
\usepackage[all]{xy}
\usepackage{manfnt}
\usepackage{pigpen}
\usepackage{bm}
\usepackage{enumitem}

\usepackage{tikz}

\usepackage{tikz}
\usetikzlibrary{matrix}
\usetikzlibrary{positioning, calc}
\usetikzlibrary{shapes}
\usetikzlibrary{arrows}  
\usetikzlibrary{calc,3d}
\usetikzlibrary{decorations,decorations.pathmorphing}
\usetikzlibrary{through}
\tikzset{ext/.style={circle, draw,inner sep=1pt},int/.style={circle,draw,fill,inner sep=1pt},nil/.style={inner sep=1pt}}
\tikzset{exte/.style={circle, draw,inner sep=3pt},inte/.style={circle,draw,fill,inner sep=3pt}}
\tikzset{diagram/.style={matrix of math nodes, row sep=3em, column sep=2.5em, text height=1.5ex, text depth=0.25ex}}
\tikzset{diagram2/.style={matrix of math nodes, row sep=0.5em, column sep=0.5em, text height=1.5ex, text depth=0.25ex}}
\usepackage{fancyhdr}

\numberwithin{equation}{section}

\begin{document}

\title[On the wheeled PROP of stable cohomology of $Aut(F_n)$]{On the wheeled PROP of stable cohomology of $Aut(F_n)$ with bivariant coefficients}
\date{}
\author{Nariya Kawazumi}
\address{Department of Mathematical Sciences, University of Tokyo,3-8-1 Komaba, Meguro-ku, Tokyo 153-8914, Japan 
}
\email{kawazumi@ms.u-tokyo.ac.jp}
\author{Christine Vespa}
\address{Universit\'e de Strasbourg, CNRS, Institut de Recherche Math\'ematique Avanc\'ee, IRMA UMR 7501, F-67000, Strasbourg, France.}
\email{vespa@math.unistra.fr}
\subjclass[2000]{}

\newtheorem{THM}{Theorem}
\newtheorem{PROP}[THM]{Proposition}
\newtheorem{CONJ}[THM]{Conjecture}
\newtheorem{COR}{Corollary}
\newtheorem{thm}{Theorem}[section]
\newtheorem{prop}[thm]{Proposition}
\newtheorem{propdef}[thm]{Proposition-Definition}
\newtheorem{cor}[thm]{Corollary}
\newtheorem{lm}[thm]{Lemma}
\newtheorem{conj}[thm]{Conjecture}

\theoremstyle{definition}
\newtheorem{defn}[thm]{Definition}
\newtheorem{ex}[thm]{Example}
\newtheorem{EXAM}{Example}
\newtheorem{exer}[thm]{Exercise}
\theoremstyle{remark}
\newtheorem{rem}[thm]{Remark}
\newtheorem{REM}{Remark}
\newtheorem{nota}[thm]{Notation}
\newtheorem{hyp}[thm]{Hypothesis}
\newtheorem{quest}{Question}
\renewcommand{\phi}{\varphi}
\renewcommand{\epsilon}{\varepsilon}
\renewcommand{\hom}{\mathrm{Hom}}
\newcommand{\cre}{\mathrm{cr}}
\newcommand{\F}{\mathcal{F}}
\newcommand{\C}{\mathcal{C}}
\newcommand{\V}{\mathcal{V}}
\newcommand{\N}{\mathbb{N}}
\newcommand{\Z}{\mathbb{Z}}
\newcommand{\Q}{\mathbb{Q}}

\newcommand{\abe}{\mathfrak{a}}
\newcommand{\gr}{\mathbf{gr}}
\newcommand{\kk}{{\Bbbk}}

\newcommand{\Hst}{\mathcal{H}}
\newcommand{\Ext}{\mathcal{E}}
\newcommand{\Sur}{\mathcal{S}}

\begin{abstract}
We show that the stable cohomology of automorphism groups of free groups with coefficients obtained by applying $Hom(-, - )$ to tensor powers of the abelianization, is equipped with the structure of a wheeled PROP $\mathcal{H}$. We  define another wheeled PROP $\Ext$ by $Ext$-groups in the category of functors from the category of finitely generated free groups to $\kk$-modules. The main result of this paper is the construction of a morphism of wheeled PROPs $\phi: \Ext \to \Hst$ such that $\phi(\Ext)$ is the wheeled PROP generated by the cohomology class $h_1$ constructed by the first author.
\end{abstract}

\maketitle

\section{Introduction}
This paper concerns the cohomology of automorphism groups of free groups $Aut(\Z^{*n})$, for $n \in \N$,  with coefficients given 
by the  $\kk$-modules:
$$B_{l,q}(\Z^{*n}, \Z^{*n}):=Hom_{\mathcal{V}}((\kk^n)^{\otimes l}, (\kk^n)^{\otimes q} )$$
where $l,q \in \N$, $\kk$ is a commutative ring and $\V$ is the category of $\kk$-modules and where the structure of $Aut(\Z^{*n})$-module on $B_{l,q}(\Z^{*n}, \Z^{*n})$ is given by the diagonal action.

In \cite{Kawa-Magnus} (see also \cite{Kawa-TMM}), for $n \geq 2$, the first author introduced a non-zero cohomology class:
$$h_1 \in H^1(Aut(\Z^{*n}), 
Hom_{\mathcal{V}}(\kk^n, (\kk^n)^{\otimes 2}))$$ 
 and constructed, from $h_1$, cohomology classes $h_p \in H^p( Aut(\Z^{*n}), Hom_{\mathcal{V}}(\kk^n, (\kk^n)^{\otimes p+1}))$ for $p>1$ and $\overline{h_p} \in H^p( Aut(\Z^{*n}), (\kk^n)^{\otimes p})$ for $p\geq 1$, even in the unstable range. The construction of these classes are inspired by previous works by Morita \cite{Mo95} and Morita with the first author \cite{KM96, KM01} concerning cohomology classes of the mapping class group with trivial coefficients $\mathbb{Q}$. (See Remark \ref{Rem-M-KM}.)

By classical construction (see Section \ref{def-stab}) there are group morphisms:
$$H^*(Aut(\Z^{*n+1}); B_{l,q}(\Z^{*n+1}, \Z^{*n+1})) \xrightarrow{\alpha_n} H^*(Aut(\Z^{*n}); B_{l,q}(\Z^{*n}, \Z^{*n})).$$

The stable cohomology of the automorphism groups of free groups with coefficients given by $B_{l,q}$ is defined  by
$$H^*_{st}(B_{l,q}):=\underset{n \in \N}{\text{lim}}H^*(Aut(\Z^{*n}); B_{l,q}(\Z^{*n}, \Z^{*n}))$$
where the limit is taken over the group morphisms $\alpha_n$. 

By a result of Randal-Williams and Wahl \cite{RWW}, this cohomology stabilizes so that the stable cohomology $H^i_{st}(B_{l,q})$ is isomorphic to $H^i(Aut(\Z^{*n}); B_{l,q}(\Z^{*n}, \Z^{*n}) )$ for $n$ big enough. It follows from the stability that stable cohomology is equipped with a cup product map:
$$\cup : H^*_{st}(B_{l_1,q_1})\otimes H^*_{st}(B_{l_2,q_2})\to  H^{*}_{st}(B_{l_1,q_1}\otimes B_{l_2,q_2} )$$
for $l_1,q_1, l_2, q_2 \in \N$.

In Definition \ref{def-H}, we define the PROP $\Hst$ where the morphisms are the graded $(\mathfrak{S}_q, \mathfrak{S}_l)$-bimodules 
$$\Hst(q,l)=H^*_{st}(B_{l,q})$$
where the action of symmetric group $\mathfrak{S}_q$ (resp. $\mathfrak{S}_l$) is given by place permutation on the tensor product $(-)^{\otimes q}$ (resp. $(-)^{\otimes l}$) and where the horizontal composition is given by the cup product map for stable cohomology and the vertical composition is induced by the composition in $\V$.

We show that this PROP is equipped with further structure:

\begin{PROP}[Proposition \ref{H-wP}]
The PROP $\Hst$ is a wheeled PROP i.e. it is equipped with contraction maps
$$\xi_j^i: \Hst(q,l) \to \Hst(q-1, l-1)$$
for $1\leq i\leq q$ and $1\leq j\leq l$ compatible with the structure of PROP.
\end{PROP}
Wheeled PROPs were introduced by Markl, Merkulov and Shadrin in \cite{MMS} to treat PROPs equipped with trace maps. The typical example of a wheeled PROP is the PROP of endomorphism of a free finitely-generated module where the contractions are given by partial trace maps (see Example \ref{ex-wP}).
The wheeled PROP structure on the PROP $\Hst$ should be viewed as a cohomological version of the wheeled endomorphism PROP.

In the stable range, for $p>1$, the cohomology classes $h_p$ are obtained from $h_1$ using the horizontal and vertical composition in the PROP $\Hst$ and for $p\geq 1$, the classes $\overline{h_p}$ are obtained from $h_p$ using the contraction maps. We deduce that the classes $h_p$ and $\overline{h_p}$ are in the sub-wheeled PROP $\mathcal{K}$ of $\Hst$ generated by the class $h_1$.

Understanding the stable cohomology of $Aut(\Z^{*n})$ with coefficients given by $B_{l,q}$
is equivalent to give a description of the wheeled PROP $\Hst$ in terms of generators and relations.
This is open; in particular, it is unknown whether the inclusion functor $\mathcal{K} \hookrightarrow \Hst$ is strict.

By the results in  \cite{DV2} and \cite{Djament} we know that the stable cohomology of $Aut(\Z^{*n})$ with non-constant coefficients is closely related to $Ext$-groups in the category $\F(\gr; \kk)$ of functors from the category $\gr$ of finitely generated free groups to the category $\V$ of $\kk$-modules. More precisely, the main result of \cite{DV2}, obtained using functor homology methods, implies that $\Hst(0,l)=0$ for $l>0$ and  \cite[Th\'eor\`eme 4]{Djament} gives, for $\kk=\Q$ a natural isomorphism
\begin{equation} \label{iso-Dja}
\Hst(q,0) \simeq {\underset{j\in \N}{\bigoplus}}Ext^{*-j}_{\F(\gr; \kk)}(\Lambda^j  \abe, \abe^{\otimes q})
\end{equation}
where $\abe^{\otimes q}$ is the $q$-th tensor power of the abelianization functor and $\Lambda^j  \abe$ is the $j$-th exterior power of the abelianization functor (see Section \ref{Rappel-foncteurs}).
The $Ext$-groups on the right hand side of the isomorphism (\ref{iso-Dja}) are computed  in \cite{Vespa} (the result is recalled in Proposition \ref{Ext2}) giving the explicit computation of $\Hst(q,0)$. 

Note that Randal-Williams obtained in \cite{RW} the computation of $\Hst(0,l)$ and $\Hst(q,0)$ using geometric techniques independent of the approach in  \cite{DV2} and \cite{Djament}.

Inspired by a conjecture given in \cite{Djament}  we define in  Section \ref{PROP-E}, for $\kk=\Q$,  a  PROP $\Ext$ where the morphisms are the  graded $(\mathfrak{S}_q, \mathfrak{S}_l)$-bimodules 
$$\Ext(q,l)={\underset{j\in \N}{\bigoplus}}Ext^{*-j}_{\F(\gr; \kk)}(\abe^{\otimes l} \otimes \Lambda^j  \abe, \abe^{\otimes q}).$$
The subPROP $\Ext_0$ of $\Ext$ given by 
$$\Ext_0(q,l)=Ext^*_{\F(\gr; \kk)}(\abe^{\otimes l} , \abe^{\otimes q})$$
has been studied by the second author: in particular, by \cite[Proposition 3.5]{Vespa} the PROP $\Ext_0$ is generated by its underlying operad $\mathcal{P}_0$. In Proposition \ref{operade-gr} we give an explicit description of this operad by generators and relations. A more conceptual description of the operad  $\mathcal{P}_0$ is the following:

\begin{PROP}[Proposition \ref{operad-sus}]
The operad $\mathcal{P}_0$ is the operadic suspension of the operad $\mathcal{C}om$ of  \textit{non-unital} commutative algebras.
\end{PROP}

The previous results on $\Ext_0$ and $\mathcal{P}_0$ are also true for $\kk=\Z$.

The forgetful functor from wheeled PROP to operads has a left adjoint. We denote by
$\C_{\mathcal{P}_0^{\circlearrowright}}$ the wheeled PROP associated to the operad $\mathcal{P}_0$ by this functor.
We obtain the following result
\begin{PROP} [Proposition \ref{PROP-E-explicite}]  \label{Prop2-intro}
There is an isomorphism of PROPs
$$\chi: \C_{\mathcal{P}_0^{\circlearrowright}} \xrightarrow{\simeq} \Ext.$$

\end{PROP}
In particular, $\Ext$ inherits a structure of \textit{wheeled} PROP via this isomorphism. The existence of a wheeled structure on the PROP $\Ext$ is quite surprising and is very specific to the situation studied in this paper (see Remark \ref{w-on-E}).
We deduce from Proposition \ref{Prop2-intro} a description of the wheeled PROP $\Ext$ by generators and relations.

The main result of this paper is the following :
\begin{THM} [Theorem \ref{phi}]
There is a morphism of wheeled PROPs 
$$\phi: \Ext \to \Hst$$ 
such that $\phi(\Ext) \simeq \mathcal{K}$.
\end{THM}

Let $\Ext_w$ (resp. $\Hst'$) be the subPROP of $\Ext$ (resp. $\Hst$) keeping only the morphisms to $0$ and the endomorphisms in degree $0$ in $\Ext$ (resp. $\Hst$).  Djament's result can be rephrased in the following way:
\begin{PROP} \cite[Th\'eor\`eme 4]{Djament} 
By restriction, $\phi$ induces an isomorphism of PROPs: $\phi': \Ext_w \xrightarrow{\simeq} \Hst'$. 
\end{PROP}

In \cite[Th\'eor\`eme 7.4]{Djament} Djament conjectures  that there exist  graded isomorphisms: 
$$\overset{q-l}{\underset{j=0}{\bigoplus}}s^{-j}Ext^*_{\F(\gr; \kk)}(\abe^{\otimes l} \otimes \Lambda^j  \abe, \abe^{\otimes q}) \simeq H^*_{st}(Hom_{\mathcal{V}}(\abe^{\otimes l}, \abe^{\otimes q} )).$$

Natural candidates for maps giving these isomorphisms are the maps 
$$
\phi_{q,l}: \overset{q-l}{\underset{j=0}{\bigoplus}}s^{-j}Ext^*_{\F(\gr; \kk)}(\abe^{\otimes l} \otimes \Lambda^j  \abe, \abe^{\otimes q}) \to H^*_{st}(Hom_{\mathcal{V}}(\abe^{\otimes l}, \abe^{\otimes q} ))
$$
given by the functor $\phi$. By functoriality these maps are compatible with horizontal and vertical compositions in the PROPs and with the contractions.

Djament's conjecture can be rephrased in the following way:

\begin{CONJ}
The morphism $\phi$ is an isomorphism of wheeled PROPs.
\end{CONJ}

This conjecture would imply that the inclusion functor $\mathcal{K} \hookrightarrow \Hst$ is an equivalence of categories i.e. the class $h_1$ would generate the wheeled PROP $\Hst$.

\vspace{.3cm}
\textbf{Acknowledgements: }
The authors are grateful to the CNRS and the JSPS for their support through the project "Cohomological study of mapping class groups and related topics" managed by the two authors. The results of this paper form part of this project.

The authors would like to thank Vladimir Dotsenko for suggesting using \textit{wheeled} PROPs and operads after a talk on this work given in Strasbourg, where the results were given in terms of PROPs and operads. They are also grateful to Geoffrey Powell for useful comments on previous versions of this paper, in particular on Sections \ref{explicit-classes} and \ref{operad-P0}.

The second author would like to warmly thank  the University of Tokyo for the invitation between April and June 2019 where this project was achieved and where part of this paper was written.

The first author was supported in part by the grant JSPS KAKENHI 18KK0071, 20H00115, 18K03283 and 19H01784.
The second author was partially supported by the ANR Project Chrok, ANR-16-CE40-0003, the ANR Project AlMaRe ANR-19-CE40-0001-01 and the ANR Project HighAGT ANR-20-CE40-0016.
\tableofcontents

\textbf{Notation:} We denote by $\mathbb{N}=\{0, 1, \ldots \}$ and by $\textbf{q}$ the set $\{1, \ldots, q\}$.\\
In all the paper, $\kk$ is a commutative ring which will be, most of the time, a field of characteristic zero or $\Z$. We denote by $\mathcal{V}$ the category of $\kk$-modules and $\mathcal{V}^f$ its full subcategory of free finitely-generated modules.

A \textit{homological} $\Z$-graded $\kk$-module is denoted by  $V_\bullet=\bigoplus _nV_n$ and a morphism \textit{of homological degree $d$}: $f:V_\bullet \to W_\bullet$ is a family of linear maps $f_n: V_n \to W_{n+d}$ for all $n \in \Z$. To $V_\bullet$ a homological $\Z$-graded $\kk$-module, we associate a \textit{cohomological} $\Z$-graded $\kk$-module $V^\bullet$ by $V^n:=V_{-n}$. A morphism of homological degree $d$ 
corresponds to a morphism of \textit{cohomological degree $-d$}. 

Graded $\kk$-modules and morphisms of degree $0$ form a category denoted by $gr\mathcal{V}$. 
For $\otimes $ the tensor product of $\Z$-graded $\kk$-modules, the category $(gr\mathcal{V}, \otimes, \kk)$ is equipped with the symmetry given by the maps $\tau: V \otimes W \to W\otimes V$ defined by $\tau(v \otimes w):=(-1)^{pq} w \otimes v$ where $v \in V_p$ and $w \in W_q$.

For $V$ a $\kk$-module we denote again by $V$ the graded $\kk$-module concentrated in degree $0$ where it is equal to $V$.

Let $\kk s$ be the graded $\kk$-module concentrated in degree one and such that $(\kk s)_1$ is spanned by $s$; the \textit{suspension} of a graded $\kk$-module $V$ is $sV:=\kk s \otimes V$ so that  $(sV)_i=V_{i-1}$.

The duality functor, denoted by $-^*: (\mathcal{V})^{op} \to \mathcal{V}$ is defined by $Hom_{\mathcal{V}}(-,\kk)$.

Non-specified tensor products are taken over $\kk$.

For $C,C'$ objects of a category $\C$, the set of morphisms from $C$ to $C'$ is denoted by $Hom_\C(C,C')$ or $\C(C,C')$.

\part{Recollections}
\section{Recollections on PROPs and operads}
PROPs and operads arose in the work of Mac Lane \cite{ML65}. Since then, they have turned out to be very important algebraic structures, especially in algebraic topology.

In this section we recall some basic facts that we will use in the paper on  PROPs and operads, as well as  their wheeled versions introduced more recently by Markl, Merkulov and Shadrin  \cite{MMS}.

\subsection{Classical PROPs and operads}

For PROPs, we refer the reader to \cite{Markl} for further details. The notion of PROP is closely related to the notion of operad. For operads, we refer the reader to \cite{LV}.

A PROP is a symmetric monoidal category $(\C, \otimes, 1)$ with objects the natural numbers whose symmetric monoidal structure $\otimes$ is given by the sum of integers on objects.

In this paper we will consider PROPs enriched over $gr\mathcal{V}$, called graded linear PROPs. Such a PROP $\C$ is a collection $\{\C(m,n)\}_{m,n\in \mathbb{N}}$ of graded $(\mathfrak{S}_m, \mathfrak{S}_n)$-bimodules (i.e. graded left $\mathfrak{S}_m \otimes (\mathfrak{S}_n)^{op}$-modules) together with two types of compositions, the horizontal composition
$$\C(m_1,n_1)\otimes \C(m_2,n_2) \to \C(m_1+m_2,n_1+n_2),$$
induced by the monoidal product, and the vertical composition
$$ \C(n,l) \otimes \C(m,n) \to \C(m,l)$$
given by the categorical composition.

An \textit{operation of biarity} $(m,n)$ in a PROP $\C$ is an element in $\C(m,n)$.

In the rest of this paper all the operads and PROPs will be \textit{graded linear}. To simplify the terminology we will call them simply operads and PROPs.

An important example of PROP is the endomorphism PROP of a graded $\kk$-module:
\begin{ex} \label{ex-P}
To  an object $V$ of $gr\mathcal{V}$, we associate the PROP, denoted by $\mathcal{E}nd_V$, defined by: 
$$\mathcal{E}nd_V(m,n)=Hom_{gr\mathcal{V}}(V^{\otimes m}, V^{\otimes n})$$ with the action of the symmetric groups given by the action on the tensor product by place permutations. The horizontal composition is given by the tensor product of linear maps and the vertical composition  by the composition in $\mathcal{V}$. 
\end{ex}

Every  PROP $\C$ has an underlying  operad $\mathcal{P}_\C$ given by $\mathcal{P_\C}(n)=Hom_\C(n,1)$.

Conversely, every operad $\mathcal{P}_0$ generates a PROP $\C_{\mathcal{P}_0}$ where 
$$\C_{\mathcal{P}_0}(q,l)= \underset{f: \textbf{q} \to \textbf{l}}{\bigoplus} \ \overset{l}{\underset{i=1}{\bigotimes }}\mathcal{P}_0(f^{-1}(i)).$$

For two PROPs, $\C$ and $\C'$, a morphism of PROPs is a strict monoidal functor $F: \C \to \C'$ which is the identity on the objects and graded linear (i.e. the maps between the Hom-sets are morphisms of degree $0$).

For a PROP $\C$, a morphism in $\C(m,n)$  can be represented by a \textit{directed $(m,n)$-graphs} i.e. finite, not necessary connected, graph such that each edge is equipped with a direction, there are no directed cycles and the set of legs is divided into the set of inputs labeled by $\{1, \ldots, m\}$ and outputs labeled by  $\{1, \ldots, n\}$.

In our pictures the graphs are oriented from top to bottom.

Using the horizontal composition in $\C$, each morphism in $\C$ is the disjoint union of \textit{$(m,n)$-corollas} which are (connected) graphs of the form:

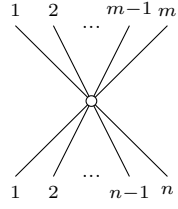
\[
	\vcenter{\hbox{\begin{tikzpicture}[baseline=1.8ex,scale=0.5]
	\draw (0,4) -- (2,2);
	\draw (1,4) -- (2,2);
	\draw (3,4) -- (2,2);
	\draw (4,4) -- (2,2);
	\draw (0,0) -- (2,2);
	\draw (1,0) -- (2,2);
	\draw (3,0) -- (2,2);
	\draw (4,0) -- (2,2);
	\draw (2,3.7) node[above] {$\scriptstyle{\ldots}$};
	\draw (2,0.3) node[below] {$\scriptstyle{\ldots}$};
	\draw (0,4) node[above] {$\scriptstyle{1}$};	
	\draw (1,4) node[above] {$\scriptstyle{2}$};	
	\draw (3,4) node[above] {$\scriptstyle{m-1}$};
	\draw (4,4) node[above] {$\scriptstyle{m}$};
	\draw (0,0) node[below] {$\scriptstyle{1}$};
	\draw (1,0) node[below] {$\scriptstyle{2}$};
	\draw (3,0) node[below] {$\scriptstyle{n-1}$};
	\draw (4,0) node[below] {$\scriptstyle{n}$};
	\draw[fill=white] (2,2) circle (4pt);
	\end{tikzpicture}}}
		\]

For example, the following depicts a morphism in $\C(5,4)$:

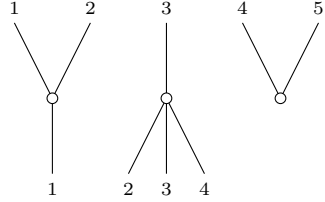
\[
	\vcenter{\hbox{\begin{tikzpicture}[baseline=1.8ex,scale=0.5]
	\draw (1,4) -- (2,2);
	\draw (3,4) -- (2,2);
	\draw (2,0) -- (2,2);
	\draw (5,2) -- (5,4);
	\draw (5,2) -- (4,0);
	\draw (5,2) -- (5,0);
	\draw (5,2) -- (6,0);
	\draw (8,2) -- (7,4);
	\draw (8,2) -- (9,4);
	\draw (1,4) node[above] {$\scriptstyle{1}$};	
	\draw (3,4) node[above] {$\scriptstyle{2}$};
	\draw (5,4) node[above] {$\scriptstyle{3}$};
	\draw (7,4) node[above] {$\scriptstyle{4}$};
	\draw (9,4) node[above] {$\scriptstyle{5}$};
	\draw (2,0) node[below] {$\scriptstyle{1}$};
	\draw (4,0) node[below] {$\scriptstyle{2}$};
	\draw (5,0) node[below] {$\scriptstyle{3}$};
	\draw (6,0) node[below] {$\scriptstyle{4}$};
	\draw[fill=white] (2,2) circle (4pt);
	\draw[fill=white] (5,2) circle (4pt);
	\draw[fill=white] (8,2) circle (4pt);
	\end{tikzpicture}}}
		 \]

\subsection{Wheeled PROPs and wheeled operads}
In this section we recall the wheeled versions of PROPs and operads, introduced by Markl, Merkulov and Shadrin in \cite{MMS}, in order to encode algebras equipped with traces. (See Example \ref{ex-wP}).

A wheeled PROP is a PROP equipped with \textit{contractions}
$$\xi_j^i: \C(m,n) \to \C(m-1, n-1)$$
for $1\leq i\leq m$ and $1 \leq j\leq n$. These contractions satisfy compatibility axioms.

For a wheeled PROP $\C$, a morphism in $\C(m,n)$  can be represented by a directed $(m,n)$-graph having possibly wheels and loops.

The contraction $\xi_j^i$ can be viewed as connecting the $i^{th}$ input and the $j^{th}$ output. For example, for a $(m,n)$-corolla we have the following picture:
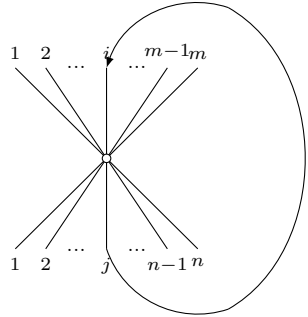
\[
	\vcenter{\hbox{\begin{tikzpicture}[baseline=1.8ex,scale=0.4]
	\draw (3,6) node{};
	\draw (3,0) node{};
	\draw (0,6) -- (3,3);
	\draw (1,6) -- (3,3);
	\draw (3,6) -- (3,3);
	\draw (5,6) -- (3,3);
	\draw (6,6) -- (3,3);
	\draw (0,0) -- (3,3);
	\draw (1,0) -- (3,3);
	\draw (3,0) -- (3,3);
	\draw (5,0) -- (3,3);
	\draw (6,0) -- (3,3);
	\draw (3,0) to[bend left=-45] (7,-2);
	\draw (7,8) to[bend left=60] (7,-2);
	\draw[->,>=latex]  (7,8) to[bend left=-45] (3,6);
	\draw (2,5.7) node[above] {$\scriptstyle{\ldots}$};
	\draw (4,5.7) node[above] {$\scriptstyle{\ldots}$};
	\draw (2,0.3) node[below] {$\scriptstyle{\ldots}$};
	\draw (4,0.3) node[below] {$\scriptstyle{\ldots}$};
	\draw (0,6) node[above] {$\scriptstyle{1}$};	
	\draw (1,6) node[above] {$\scriptstyle{2}$};	
	\draw (3,6) node[above] {$\scriptstyle{i}$};
	\draw (5,6) node[above] {$\scriptstyle{m-1}$};
	\draw (6,6) node[above] {$\scriptstyle{m}$};
	\draw (0,0) node[below] {$\scriptstyle{1}$};
	\draw (1,0) node[below] {$\scriptstyle{2}$};
	\draw (3,0) node[below] {$\scriptstyle{j}$};
	\draw (5,0) node[below] {$\scriptstyle{n-1}$};
	\draw (6,0) node[below] {$\scriptstyle{n}$};
	\draw[fill=white] (3,3) circle (4pt);
	\end{tikzpicture}}}
		\]

In a wheeled PROP, vertical composition is determined by the horizontal composition and the contractions by the following formula:
\begin{equation}\label{vertical-wP}
 \C(n,l) \otimes \C(m,n) \to \C(n+m,l+n) \xrightarrow{(\xi^1_{l+1})^n} \C(m,l)
 \end{equation}
(see  \cite[(17)]{MMS}).

A fundamental example of a wheeled PROP is the wheeled endomorphism PROP associated with a free finitely-generated $\kk$-module where the contractions are given by the trace map:

\begin{ex} \label{ex-wP}
By classical linear algebra, for objects $E$ and $F$ in $\mathcal{V}$, we have canonical homomorphisms  $E^* \otimes F \to Hom_\V(E,F)$ and $E^* \otimes F^*\to (E \otimes F)^*$ which are isomorphisms if $E$ is a free finitely-generated $\kk$-module. 

For $V$ an object of $\mathcal{V}^f$, 
the previous observations give rive to the isomorphisms:
$$\theta_{m,n}: Hom_{\mathcal{V}^f}(V^{\otimes m}, V^{\otimes n}) \xrightarrow{\simeq}(V^*)^{\otimes m} \otimes V^{\otimes n} $$

By Example \ref{ex-P}, $\mathcal{E}nd_V(m,n)=Hom_{\mathcal{V}}(V^{\otimes m}, V^{\otimes n})$ defines a PROP.
For $1 \leq i \leq m$ and $1 \leq j \leq n$, the contractions $\xi^i_j: \mathcal{E}nd_V(m,n) \to \mathcal{E}nd_V(m-1,n-1)$ correspond through the previous isomorphisms to the maps 
$$\varphi^i_j: (V^*)^{\otimes m} \otimes V^{\otimes n} \to (V^*)^{\otimes m-1} \otimes V^{\otimes n-1}$$ given by:
$$\varphi^i_j(f_1 \otimes \ldots \otimes f_m \otimes v_1 \otimes \ldots \otimes v_n)=f_i(v_j)(f_1 \otimes \ldots f_{i-1} \otimes f_{i+1} \ldots f_m \otimes v_1\otimes  \ldots  v_{j-1} \otimes v_{j+1} \otimes \ldots \otimes v_n)$$
where $f_i \in V^*$ and $v_i \in V$. 

Note that  $\varphi_1^1: V^* \otimes V \to \kk$ is the evaluation and $\varphi_1^1 \circ \theta_{1,1}: Hom_{\mathcal{V}^f}(V, V)  \to \kk$ is the  trace map $Tr$ which associates to an endomorphism of $V$ its trace. 
\end{ex}

A morphism of wheeled PROPs is a morphism of PROPs that is compatible with the contractions.

The forgetful functor from the category of wheeled PROPs to the category of PROPs has a left adjoint denoted by $(-)^{\circlearrowright}$. For $\C$ a PROP, $\C^{\circlearrowright}$ is called the \textit{wheeled completion} of $\C$ (see \cite[Definition 2.1.9]{MMS}).

Recall from \cite[Definition 5.1.1]{MMS}, that a \textit{wheeled operad} $\mathcal{P}=\{ \mathcal{P}(n,m)\}_{m,n}$, where $n\in \N$ and $m\in \{0,1\}$, consists of:
\begin{enumerate}
\item \label{wop-1} an ordinary operad $\mathcal{P}_0:=\{ \mathcal{P}(n,1) \}_{n\geq 0}$;
\item \label{wop-2} a right $\mathcal{P}_0$-module $\mathcal{P}_w:=\{ \mathcal{P}(n,0) \}_{n \geq 0}$;
\item for $1 \leq i \leq n$, contractions $\xi^i: \mathcal{P}_0(n) \to \mathcal{P}_w(n-1)$, satisfying compatibility conditions with the structures given in (\ref{wop-1}) and (\ref{wop-2}).
\end{enumerate}

The operad  $\mathcal{P}_0$ is called the \textit{operadic part} of $\mathcal{P}$ and $\mathcal{P}_w$ its \textit{wheeled} part.

Recall that the operad $\mathcal{P}_0$ is itself a right $\mathcal{P}_0$-module.

Every wheeled PROP $\C$ has an underlying wheeled operad $\mathcal{P}^\C$ where the operadic part is $\mathcal{P}^\C_0=\{Hom_\C(n,1)\}$, the wheeled part is $\mathcal{P}^\C_w=\{Hom_\C(n,0)\}$ and the contractions $\xi^i: \mathcal{P}^\C_0(n) \to \mathcal{P}^\C_w(n-1)$ are the contractions $\xi^i_1:Hom_\C(n,1) \to Hom_\C(n,0)$ of the wheeled PROP.

Conversely, every wheeled operad $\mathcal{P}$ generates a wheeled PROP $\C_\mathcal{P}$. The following explicit description of the  wheeled PROP $\C_\mathcal{P}$ follows from the description of the free wheeled PROP generated by a collection of $(\mathfrak{S}_m, \mathfrak{S}_n)$-bimodules given in \cite[Section 2.1.6]{Merkulov-2} (see also \cite[Section 2.3]{Merkulov-3}).

\begin{prop} \label{C_P}
The wheeled PROP $\C_\mathcal{P}$ associated to a wheeled operad $\mathcal{P}$ is given by the $(\mathfrak{S}_q,\mathfrak{S}_l)$-bimodules
$$\C_\mathcal{P}(q,l)= \underset{\substack{J \subset \textbf{q}\\ f: J\twoheadrightarrow \textbf{l}}}{\bigoplus} \  \underset{\substack{1 \leq k\leq | \textbf{q}\setminus J| \\ X_1 \amalg \ldots \amalg X_{k } =\textbf{q}\setminus J \\min(X_1)< \ldots <min(X_k)}}  {\bigoplus} \  \left(\overset{l}{\underset{i=1}{\bigotimes }}\mathcal{P}_0(f^{-1}(i))\right) \otimes \left(\overset{k }{\underset{i=1}{\bigotimes }}\mathcal{P}_w(X_i)\right) $$
where $J \subset \textbf{q}$, $X_1 \amalg \ldots \amalg X_{k} $ is a partition of $\textbf{q}\setminus J$ into non-empty sets such that: $min(X_1)<\ldots <min(X_k)$.

The symmetric group $\mathfrak{S}_l$ acts by postcomposition on $f: J\to \textbf{l}$ and $\mathfrak{S}_q$ by precomposition on $f: J\to \textbf{l}$ and on the partition $X_1 \amalg \ldots \amalg X_{k} $.

Horizontal composition is induced by disjoint union of maps and partitions. 

The contractions $\xi^i_j: \C_\mathcal{P}(q,l) \to \C_\mathcal{P}(q-1,l-1)$ for $1 \leq i\leq q$, $1 \leq j \leq l$ are induced by the composition in the operad $\mathcal{P}_0$, the  right $\mathcal{P}_0$-module structure on $\mathcal{P}_w$ and the contractions $\xi^i: \mathcal{P}_0(n) \to \mathcal{P}_w(n-1)$. 

More explicitly, on the direct summand of $\C_\mathcal{P}(q,l) $ corresponding to $J \subset \textbf{q}$,  $f: J\twoheadrightarrow \textbf{l}$, $1 \leq k\leq | \textbf{q}\setminus J|$ and $X_1 \amalg \ldots \amalg X_{k } =\textbf{q}\setminus J$ such that $min(X_1)<\ldots <min(X_k)$
\begin{enumerate}[label=({\roman*})]
\item if $i \in J$ and $j \in f(i)$ we apply $\xi^i:\mathcal{P}_0(f^{-1}(j)) \to \mathcal{P}_w(f^{-1}(j)\setminus \{i\})$;
\item if $i \in J$ and $j \in f(k)$ for $k \not= i$, we apply the partial composition in the operad $\mathcal{P}_0$
$$\underset{i}{\circ}:  \mathcal{P}_0\big(f^{-1}(f(i))\big) \otimes \mathcal{P}_0(f^{-1}(j)) \to  \mathcal{P}_0(f^{-1}(f(i)) \amalg f^{-1}(j) \setminus \{i\});$$
\item if $i \in X_p \subset \textbf{q}\setminus J$, we apply the right $\mathcal{P}_0$-module structure on $\mathcal{P}_w$: 
$$\underset{i}{\circ}: \mathcal{P}_w\big (X_p) \otimes \mathcal{P}_0(f^{-1}(j)) \to  \mathcal{P}_w(X_p \amalg f^{-1}(j) \setminus \{i\}).$$
\end{enumerate}
\end{prop}
To illustrate the contractions defined in the previous proposition consider in $\C_\mathcal{P}(9,2)$ the summand corresponding to $J=\{1,2,3,4,5\}$, $f: J\to \mathbf{2}$ given by $f(1)=f(2)=f(3)=1$ and $f(4)=f(5)=2$, $X_1=\{6,7,8\}$ and $X_2=\{9\}$ and consider the element $X \in \mathcal{P}_0(\{1,2,3\}) \otimes \mathcal{P}_0(\{4,5\}) \otimes \mathcal{P}_w(X_1) \otimes \mathcal{P}_w(X_2)$ given by the graph:

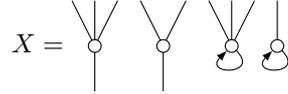
\[
	X=\vcenter{\hbox{\begin{tikzpicture}[baseline=1.8ex,scale=0.3]
	\draw (2,0) -- (2,2);
	\draw (2,4) -- (2,2);
	\draw (3,4) -- (2,2);
	\draw (1,4) -- (2,2);
	\draw (5,0) -- (5,2);
	\draw (4,4) -- (5,2);
	\draw (6,4) -- (5,2);
	\draw (7,4) -- (8,2);
	\draw (8,4) -- (8,2);
	\draw (9,4) -- (8,2);
	\draw (10,4) -- (10,2);
	\draw[fill=white] (2,2) circle (8pt);
	\draw[fill=white] (5,2) circle (8pt);
	\draw[fill=white] (8,2) circle (8pt);
	\draw[fill=white] (10,2) circle (8pt);
	\draw[->,>=latex]  (8.2,1.8) to [out=-45,in=225,looseness=10] (7.8,1.8); 
	\draw[->,>=latex]  (10.2,1.8) to [out=-45,in=225,looseness=10] (9.8,1.8); 
		\end{tikzpicture}}}
		\]
		The case $(i)$ is illustrated by:
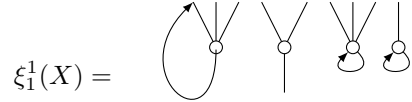
\[
	\xi_1^1(X)=\vcenter{\hbox{\begin{tikzpicture}[baseline=1.8ex,scale=0.3]
	\draw (2,4) -- (2,2);
	\draw (3,4) -- (2,2);
	\draw (1,4) -- (2,2);
	\draw (5,0) -- (5,2);
	\draw (4,4) -- (5,2);
	\draw (6,4) -- (5,2);
	\draw (7,4) -- (8,2);
	\draw (8,4) -- (8,2);
	\draw (9,4) -- (8,2);
	\draw (10,4) -- (10,2);
	\draw[fill=white] (2,2) circle (8pt);
	\draw[fill=white] (5,2) circle (8pt);
	\draw[fill=white] (8,2) circle (8pt);
	\draw[fill=white] (10,2) circle (8pt);
	\draw[->,>=latex]  (8.2,1.8) to [out=-45,in=225,looseness=10] (7.8,1.8); 
	\draw[->,>=latex]  (10.2,1.8) to [out=-45,in=225,looseness=10] (9.8,1.8); 
	\draw[->,>=latex]  (2,1.9) to [out=-90,in=225,looseness=5] (1,4); 
		\end{tikzpicture}}}
		\]	
		The case $(ii)$ is illustrated by:
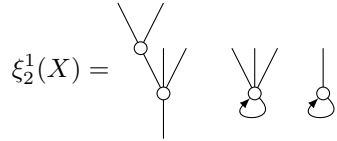
\[
	\xi_2^1(X)=\vcenter{\hbox{\begin{tikzpicture}[baseline=1.8ex,scale=0.3]
	\draw (2,4) -- (3,2);
	\draw (3,4) -- (3,2);
	\draw (4,4) -- (3,2);
	\draw (3,0) -- (3,2);
	\draw (2,4) -- (1,6);
	\draw (2,4) -- (3,6);
	\draw (6,4) -- (7,2);
	\draw (7,4) -- (7,2);
	\draw (8,4) -- (7,2);
	\draw (10,4) -- (10,2);
	\draw[fill=white] (2,4) circle (8pt);
	\draw[fill=white] (3,2) circle (8pt);
	\draw[fill=white] (7,2) circle (8pt);
	\draw[fill=white] (10,2) circle (8pt);
	\draw[->,>=latex]  (7.2,1.8) to [out=-45,in=225,looseness=10] (6.8,1.8); 
	\draw[->,>=latex]  (10.2,1.8) to [out=-45,in=225,looseness=10] (9.8,1.8); 
			\end{tikzpicture}}}
				\]
						The case $(iii)$ is illustrated by:
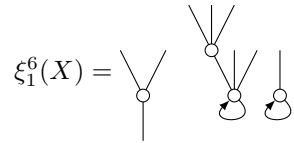
\[
	\xi_1^6(X)=\vcenter{\hbox{\begin{tikzpicture}[baseline=1.8ex,scale=0.3]
	\draw (4,0) -- (4,2);
	\draw (3,4) -- (4,2);
	\draw (5,4) -- (4,2);
	\draw (7,4) -- (8,2);
	\draw (7,4) -- (6,6);
	\draw (7,4) -- (7,6);
	\draw (7,4) -- (8,6);
	\draw (8,4) -- (8,2);
	\draw (9,4) -- (8,2);
	\draw (10,4) -- (10,2);
	\draw[fill=white] (4,2) circle (8pt);
	\draw[fill=white] (8,2) circle (8pt);
	\draw[fill=white] (7,4) circle (8pt);
	\draw[fill=white] (10,2) circle (8pt);
	\draw[->,>=latex]  (8.2,1.8) to [out=-45,in=225,looseness=10] (7.8,1.8); 
	\draw[->,>=latex]  (10.2,1.8) to [out=-45,in=225,looseness=10] (9.8,1.8); 
		\end{tikzpicture}}}
		\]
\begin{rem}
Note that
$$\C_\mathcal{P}(n,n)= \underset{ f \in \mathfrak{S}_n}{\bigoplus}(\mathcal{P}_0(1))^{\otimes n}.$$
Considering the identity operation $\mathrm{Id} \in \mathcal{P}_0(1)$ we obtain a monomorphism of $\mathfrak{S}_n$-bimodule:
$$\kk[\mathfrak{S}_n] \to \C_\mathcal{P}(n,n).$$

\end{rem}
\begin{rem}
Recall that from (\ref{vertical-wP}) vertical composition in a wheeled PROP is induced by horizontal composition and contractions. 
\end{rem}

The forgetful functor from the category of wheeled operads to the category of operads has a left adjoint denoted by $(-)^{\circlearrowright}$. For $\mathcal{P}_0$ an operad, $(\mathcal{P}_0)^{\circlearrowright}$ is called the \textit{wheeled completion} of $\mathcal{P}_0$.

\begin{rem} \label{subPROP}
The  wheeled PROP $\C_\mathcal{P}$ generated by a wheeled operad $\mathcal{P}$ has two distinguished subPROPs:
\begin{enumerate}
\item the subPROP $\C_{\mathcal{P}_0}$ generated by the operad $\mathcal{P}_0$, corresponding to forgetting the wheeled part of $\mathcal{P}$;
\item the subPROP, denoted by $\C_{w}$, such that, for all $n \in \N$, 
$$\C_{w}(n, 0)=\C_\mathcal{P}(n, 0); \qquad \C_{w}(n, n)=\kk[\mathfrak{S}_n]  \mathrm{\ for\ } n\geq 1; \qquad \C_{w}(n, m)=0 \mathrm{\ for\ } m \notin\{0,n\}$$
corresponding to forgetting the operadic part in the PROP $\C_\mathcal{P}$.
\end{enumerate}
\end{rem}

\section{Recollections on the functor category $\F(\gr; \kk)$} \label{Rappel-foncteurs}

The purpose of this section is to review results on covariant and contravariant functors from the category $\gr$ of finitely-generated free groups to $\kk$-modules. We refer the reader to \cite{HPV}, \cite{DV2} and \cite{DPV} for more details.

We denote by $\Z^{*n}$ the free group on $n$ generators.
The category $\gr$ has a skeleton with  objects $\Z^{*n}$ for $n \in \N$. Consequently, $\gr$ is essentially small  and we denote by $\F(\gr; \kk)$ (resp. $\F(\gr^{op}; \kk)$) the category of covariant (resp. contravariant) functors from $\gr$ to $\V$.

A fundamental example of functor in $\F(\gr; \kk)$ is the abelianization functor $\abe: \gr \to \V$ that sends a free group $G$ to $(G/[G,G])\otimes_\Z \kk$. 

Composing $\abe$ with the duality functor $-^*: \mathcal{V} \to (\mathcal{V})^{op}$ we obtain a functor from $\gr$ to $\V^{op}$. The category of functors from $\gr$ to $\V^{op}$ is equivalent to $(\F(\gr^{op}; \kk))^{op}$. We will denote by $\abe^\vee: \gr^{op} \to \V$ the functor corresponding to $-^* \circ \abe$ by this equivalence.

By the Yoneda lemma, the category $\F(\gr; \kk)$ has enough projective objects and a set of projective generators is given by the functors, for  $n \in \N$:
$$P_n:=\kk[\gr(\Z^{*n},-)]$$
where $\kk[-]$ is the linearization functor from the category of sets to $\V$.

Each functor $F: \gr \to \V$ can be decomposed naturally  as a direct sum $F = F(0) \oplus \bar{F}$ where $F(0)$ is the constant functor equal to $F(0)$ and $\bar{F}$ is a \textit{reduced} functor i.e. $\bar{F}(0)=0$. For simplicity we denote $\bar{P}:= \overline{P_1}$.

The notion of cross-effects and polynomial functors introduced by Eilenberg and Mac Lane for categories of modules can be extended to functors on $\gr$ and on $\gr^{op}$. The $d$th cross-effect defines an exact functor $cr_d: \F(\gr; \kk) \to \F(\gr^{\times n}; \kk)$ where $\F(\gr^{\times n}; \kk)$ is the category of functors from $\gr^{\times n}$ to $\V$. A functor $F: \gr \to \V$ is polynomial of degree $d$ if $cr_{d+1}(F)=0$ and $cr_d(F) \not=0$.
Similarly, we can define polynomial functors on $\gr^{op}$.

The functors $\abe$ and $\abe^\vee$ are reduced polynomial functors of degree one.

The reduced functor $\bar{P}$ and the cross-effects are related by the following result:

\begin{prop} \label{Prop-barP}
For $d \in \mathbb{N}$ and $F \in \F(\gr; \kk)$ there is a natural isomorphism
$$Hom_{\F(\gr; \kk)}(\bar{P}^{\otimes d}, F) \simeq cr_d(F)(\Z, \ldots, \Z).$$
\end{prop}
We deduce the following Corollary:
\begin{cor} \label{Cor-barP}
For $d \in \mathbb{N}$ and $F \in \F(\gr; \kk)$ a polynomial functor of degree $<d$ then
$$Hom_{\F(\gr; \kk)}(\bar{P}^{\otimes d}, F) =0.$$
\end{cor}

Since the abelianization functor $\abe$ takes its values in $\V^f$, for $F$ a functor from $\V^f$ to $\V$, we can postcompose $\abe$ with $F$ to obtain a functor of $\F(\gr; \kk)$. 
An important example of functor from $\V^f$ to $\V$ is the $d$th tensor product functor $T^d: \V^f \to \V$, for $d \in \N$, defined  on objects by $V \mapsto V^{\otimes d}$. The symmetric group $\mathfrak{S}_d$ acts by place permutations on $T^d$. 
The functor $\abe^{\otimes d}:= T^d \circ \abe$ is a polynomial covariant functor of degree $d$ and $(\abe^{\vee})^{\otimes d}:= T^d \circ \abe^\vee$ is a polynomial contravariant functor of degree $d$. The notion of exponential functors is a powerful tool for computation (see \cite{FFSS}). A graded exponential functor is a monoidal functor from $(\V^f, \oplus, 0)$ to $(gr\V, \otimes, \kk)$.

If $\kk$ is a field of characteristic 0, the $d$th exterior power functor is defined, on a vector space $V$,  by $\Lambda^d(V)=(T^d(V) \otimes sgn_n)_{\mathfrak{S}_d}$ where $sgn_n$ is the signature module and $\mathfrak{S}_d$ acts diagonally.  The functor $\Lambda^d$ is a direct summand of $T^d$. The functor $\Lambda^d \abe:=\Lambda^d \circ \abe$ is a polynomial covariant functor of degree $d$. The exterior powers define a graded exponential functor $\Lambda^\bullet$. In particular, there are natural commutative products $\Lambda^i \otimes \Lambda^j \to \Lambda^{i+j}$ and cocommutative coproducts $\Lambda^{i+j} \to \Lambda^i \otimes \Lambda^j$, for $i,j \in \N$. 
Composing with the abelianization functor, we obtain a natural transformation of functors in $\F(\gr, \kk)$
\begin{equation} \label{Exterior}
\Lambda^{i+j}  \abe \to \Lambda^i  \abe \otimes \Lambda^j \abe
\end{equation}
which will be used later.

For $G,F \in \F(\gr; \kk)$, the \textit{exterior tensor product} of $G$ and $F$ is the functor $G \boxtimes F: \gr \times \gr\to \kk\mathrm{-}Mod$ given on objects by
$$(G \boxtimes F)(\Z^{*n},\Z^{*m} )=G(\Z^{*n}) \otimes F(\Z^{*m}).$$

Similarly, for $G \in \F(\gr^{op}; \kk)$ and $F \in \F(\gr; \kk)$ we define $G \boxtimes F: \gr^{op}  \times \gr\to \kk\mathrm{-}Mod$.


We denote by $\amalg_d:  \gr^{\times d} \to \gr$ the functor obtained by iteration of the free product (which is the coproduct in $\gr$) and $\delta_d: \gr \to \gr^{\times d}$ the diagonal functor. The functor $\delta_d$ is right adjoint to the functor $\amalg_d$. It follows that the functor   $\delta_d^*:\F(\gr^{\times d}; \kk) \to \F(\gr;\kk)$ given by precomposition is left adjoint of the functor  $\pi_d^*: \F(\gr; \kk) \to \F(\gr^{\times d }; \kk)$ given by precomposition. 

Tensor powers  $T^\bullet$ do not define an exponential functor but we have a similar property using induction of symmetric groups (see \cite[(3)]{Vespa}). In particular we have:
\begin{equation} \label{tensor2}
\pi_2^*(\abe^{\otimes q}) \simeq  \underset{J \subset \textbf{q}}{\bigoplus}\abe^{\otimes |J|}  \boxtimes \abe^{\otimes |\textbf{q} \setminus J|} 
\end{equation}

\part{Stable cohomology of $Aut(\Z^{*n})$ with coefficients given by a bifunctor} \label{st-hom}

\section{Definition of stable homology of $Aut(\Z^{*n})$} \label{def-stab}

 Let $I_n: Aut(\Z^{*n}) \to Aut(\Z^{*n+1})$ be the group monomorphism induced by $-*\Z$. By restriction along $I_n$ we obtain a functor $U^{I_n}: Aut(\Z^{*n+1})\mathrm{-}Mod \to Aut(\Z^{*n})\mathrm{-}Mod$ where $Aut(\Z^{*n})\mathrm{-}Mod$ is the category of modules over $Aut(\Z^{*n})$.

For $B: \gr^{op} \times \gr \to \mathcal{V}$  a functor and $n\in \N$, $B(\Z^{*n}, \Z^{*n}) $ is a $Aut(\Z^{*n})^{op} \times Aut(\Z^{*n})$-module.
Let $p_n: \Z^{*n+1} \to \Z^{*n}$ be the group epimorphism  given by the projection on the first $n$ variables and $i_n: \Z^{*n} \to \Z^{*n+1}$ be the group monomorphisms given by the inclusion of the first $n$ variables.

The previous data give rise to $Aut(\Z^{*n})$-homomorphisms
$$U^{I_n}\left( B(\Z^{*n+1}, \Z^{*n+1})\right) \xrightarrow{B(i_n, p_n)} B(\Z^{*n}, \Z^{*n}) $$
where the structure of $Aut(\Z^{*n})$-module on $B(\Z^{*n}, \Z^{*n})$ and $U^{I_n}\left( B(\Z^{*n+1}, \Z^{*n+1}) \right)$ is given by the diagonal action.

This gives group morphisms:
$$H^*(Aut(\Z^{*n+1}); B(\Z^{*n+1}, \Z^{*n+1})) \xrightarrow{\alpha_n} H^*(Aut(\Z^{*n}); B(\Z^{*n}, \Z^{*n})).$$
The stable cohomology of the automorphism groups of free groups with coefficients given by $B$ is then defined  by
$$H^*_{st}(B):=\underset{n \in \N}{\text{lim}}H^*(Aut(\Z^{*n}); B(\Z^{*n}, \Z^{*n}))$$
where the limit is taken over the group morphisms $\alpha_n$. 

In this paper we consider the family of coefficients $B_{l,q}=(\abe^\vee)^{\otimes l} \boxtimes \abe^{\otimes q}$ where $l,q \in \N$ and  $\boxtimes$ is the exterior tensor product. By the usual canonical homomorphism  $E^* \otimes F \to Hom(E,F)$ for $E,F$  two $\kk$-modules which is an isomorphism if $E$ is free finitely generated, we obtain  isomorphisms  
$$B_{l,q}(\Z^{*n}, \Z^{*m})=((\kk^n)^*)^{\otimes l} \otimes (\kk^m)^{\otimes q} \simeq Hom_{\mathcal{V}}((\kk^n)^{\otimes l}, (\kk^m)^{\otimes q} ).$$
Based on this isomorphism, the functor $B_{l,q}$ will be sometimes denoted by $Hom_{\mathcal{V}}(\abe^{\otimes l}, \abe^{\otimes q} )$.

\section{Stability}
Let $\mathcal{G}$ be the category having as objects the finitely-generated free groups and where a morphism from $A$ to $B$ is a pair $(u,H)$ where $u: A \hookrightarrow B$ is an injective homomorphism and $H$ is a subgroup of $B$ such that $B=H * u(A)$. Recall from \cite[D\'efinition 4.2]{DV2} that there is a functor $\iota: \mathcal{G} \to \gr^{op} \times \gr$ sending an object $A$ to $(A,A)$ and a map $(u,H): A \to B$ to $(B=H * u(A) \to u(A) \xrightarrow{u^{-1}} A  , u: A \to B)$. 

The category $\mathcal{G}$ is homogeneous in the sense of \cite[Definition 1.3]{RWW} and the functor $B_{l,q}=(\abe^\vee)^{\otimes l} \boxtimes \abe^{\otimes q}$ is the exterior product between  a polynomial contravariant functor of degree $l$ and a polynomial covariant functor of degree $q$, so the composition $B \circ \iota$ is a coefficient system of degree $l+q$. Hence, by \cite[Theorem 5.4]{RWW}, for $i \in \N$, the group morphism
$$H^i(Aut(\Z^{*n+1}); B_{l,q}(\Z^{*n+1}, \Z^{*n+1})) \xrightarrow{\alpha_n} H^i(Aut(\Z^{*n}); B_{l,q}(\Z^{*n}, \Z^{*n}) )$$
is an isomorphism for $n\geq N_{l,q,i}$ where $N_{l,q,i}=2i+l+q+3$. We deduce that, for $n$ big enough, we have an isomorphism
\begin{equation} \label{iso-stable}
H^i_{st}((\abe^\vee)^{\otimes l} \otimes \abe^{\otimes q})\simeq H^i(Aut(\Z^{*n}); B_{l,q}(\Z^{*n}, \Z^{*n}) ).
\end{equation}

For $l_1, q_1, l_2, q_2$ in $\N$, the cup product gives morphisms:
$$\xymatrix{
H^i(Aut(\Z^{*n}); B_{l_1,q_1}(\Z^{*n}, \Z^{*n})) \otimes H^j(Aut(\Z^{*n}); B_{l_2,q_2}(\Z^{*n}, \Z^{*n}) ) 
\ar[d]^{\cup} \\ 
H^{i+j}(Aut(\Z^{*n}); B_{l_1,q_1}(\Z^{*n}, \Z^{*n}) \otimes B_{l_2,q_2}(\Z^{*n}, \Z^{*n}) )
}$$
For $n>Max(N_{l_1,q_1,i}, N_{l_2,q_2,j})$, the stability isomorphisms (\ref{iso-stable}) give the following cup product map on the stable cohomology:

 \begin{equation} \label{cup}
\cup : H^i_{st}(B_{l_1,q_1})\otimes H^j_{st}(B_{l_2,q_2})\to  H^{i+j}_{st}(B_{l_1,q_1}\otimes B_{l_2,q_2} ).
\end{equation}

\section{The wheeled PROP $\Hst$ of stable cohomology} \label{PROP-H}
The aim of this section is to prove that the stable cohomology of $Aut(\Z^{*n})$  with coefficients twisted by $B_{l,q}=(\abe^\vee)^{\otimes l} \boxtimes \abe^{\otimes q}=Hom_{\mathcal{V}}(\abe^{\otimes l}, \abe^{\otimes q} )$ defines a wheeled PROP. This should be viewed as a cohomological version of the wheeled endomorphism PROP consider in Example \ref{ex-wP}
using the stability isomorphism (\ref{iso-stable}).

Recall that for $M_1, M_2, N_1, N_2$ objects in $\V$ we have a canonical homomorphism of $\kk$-modules:
\begin{equation}\label{lambda-iso}
\lambda: Hom_\V(M_1,M_2) \otimes Hom_\V(N_1,N_2) \to Hom_\V(M_1 \otimes N_1,M_2 \otimes N_2)
\end{equation}
which is an isomorphism if $M_1, M_2, N_1, N_2$ are free finitely generated modules.
\begin{defn} \label{def-H}
The PROP $\Hst$ is defined by the graded $(\mathfrak{S}_q, \mathfrak{S}_l)$-bimodules 
$$\Hst(q,l)=H^*_{st}(Hom_{\mathcal{V}}(\abe^{\otimes l}, \abe^{\otimes q} ))$$
where the action of the symmetric group $\mathfrak{S}_q$ (resp. $\mathfrak{S}_l$) is given by place permutations of the copies of $\abe$ (resp. $(\abe^\vee$)).

The horizontal composition $\otimes: \Hst(q_1,l_1) \otimes \Hst(q_2,l_2) \to \Hst(q_1+q_2,l_1+l_2)$ is given by:
$$\xymatrix{
H_{st}^{*}(Hom_{\mathcal{V}}(\abe^{\otimes l_1}, \abe^{\otimes q_1} ))\otimes H_{st}^{*}(Hom_{\mathcal{V}}(\abe^{\otimes l_2}, \abe^{\otimes q_2} ))\ \ar[d]^{\cup}\\
H_{st}^{*}(Hom_{\mathcal{V}}(\abe^{\otimes l_1}, \abe^{\otimes q_1} )\otimes Hom_{\mathcal{V}}(\abe^{\otimes l_2}, \abe^{\otimes q_2} )) \ar[d]^-{\simeq }\\
H^{*}_{st}(Hom_{\mathcal{V}}(\abe^{\otimes l_1+l_2}, \abe^{\otimes q_1+q_2} ))
}
$$
where $\cup$ is the cup product map given in (\ref{cup}) and  the isomorphism is induced by the isomorphisms $\lambda$ given in (\ref{lambda-iso}).

The vertical composition $\circ: \Hst(l,m) \otimes \Hst(q,l)  \to \Hst(q,m)$ is given by:
$$\xymatrix{
H_{st}^{*}(Hom_{\mathcal{V}}(\abe^{\otimes m}, \abe^{\otimes l} ))\otimes H_{st}^{*}(Hom_{\mathcal{V}}(\abe^{\otimes l}, \abe^{\otimes q} ))\ \ar[d]^{\cup}\\
H_{st}^{*}(Hom_{\mathcal{V}}(\abe^{\otimes m}, \abe^{\otimes l} )\otimes Hom_{\mathcal{V}}(\abe^{\otimes l}, \abe^{\otimes q} )) \ar[d]^{\gamma}\\
H^{*}_{st}(Hom_{\mathcal{V}}(\abe^{\otimes m}, \abe^{\otimes q} ))
}
$$
where $\gamma$ is the map induced by the composition in $\V$.
\end{defn}

\begin{prop} \label{H-wP}
The PROP $\Hst$ is a wheeled PROP for the contractions, for $1\leq i\leq q$ and $1 \leq j\leq l$:
$$\xi_j^i: \Hst(q,l) \to \Hst(q-1, l-1)$$
induced, for $n \geq N_{l,q,*}$,  by the maps
$$H^*(Aut(\Z^{*n}),\varphi^i_j): H^*(Aut(\Z^{*n}), B_{l,q}(\Z^{*n}, \Z^{*n})) \to H^*(Aut(\Z^{*n}),  B_{l-1,q-1}(\Z^{*n}, \Z^{*n}))$$
where $\varphi^i_j$ are defined in Example \ref{ex-wP}
\end{prop}
\begin{proof}
For $n\geq N_{l,q,*}$, since $N_{l,q,*}>N_{l-1,q-1,*}$,  $H^*(Aut(\Z^{*n}),\varphi^i_j)$ induces a map 
$\xi_j^i: \Hst(q,l) \to \Hst(q-1, l-1)$.

We verify that the contraction maps satisfy commutativity conditions and that they are compatible with the horizontal composition.

The bi-equivariance condition for the contraction maps corresponds to the commutativity of the following diagram
$$\xymatrix{
\Hst(q,l) \ar[r]^-{\xi^i_j}  \ar[d]_{(\sigma_1, \sigma_2)} & \Hst(q-1,l-1)   \ar[d]^{(\sigma_1^{(\sigma_1^{-1}(i))}, \sigma_2^{(j)})  } \\
\Hst(q,l)  \ar[r]_-{\xi^{\sigma_1^{-1}(i)}_{\sigma_2(j)}}&\Hst(q-1,l-1) }
$$
where  $\sigma_1 \in \mathfrak{S}_q$ and $\sigma_2 \in \mathfrak{S}_l$,  $\sigma_2^{(j)}\in \mathfrak{S}_{l-1}$ is the permutation induced by $\sigma_2$ forgetting $j$ and $\sigma_2(j)$ and reindexing and $\sigma_1^{(\sigma_1^{-1}(i))} \in \mathfrak{S}_{q-1}$ is the permutation induced by $\sigma_1$ forgetting $\sigma_1^{-1}(i)$ and $i$ and reindexing.
\end{proof}

\begin{rem}
For $l>0$, $(\abe^\vee)^{\otimes l} $ is a reduced contravariant functor which is polynomial of degree $l$. It follows from the main result of \cite{DV2} that $H^{*}_{st}((\abe^\vee)^{\otimes l} )=0$ so  we have $\Hst(0,l)=0$ for $l>0$.
\end{rem}

In order to relate our results to Djament's result obtained in \cite{Djament} we introduce the following:
\begin{defn} \label{H'}
Let $\Hst'$ be the sub-PROP of $\Hst$ such that, for $n \in \N$
$${\Hst'}(n,0)={\Hst}(n,0); \qquad {\Hst'}(n,n)=\kk[\mathfrak{S}_n];  \qquad {\Hst'}(n,m)=0 \mathrm{\ if\ } m \notin \{0, n \}.$$
\end{defn}

\begin{rem}
Note that 
$$\Hst(n,n)^0=H^0_{st}(Hom_{\mathcal{V}}(\abe^{\otimes n}, \abe^{\otimes n} ))\simeq (Hom_{\mathcal{V}}((\kk^m)^{\otimes n}, (\kk^m)^{\otimes n} ))^{Aut(\Z^{*m})}$$
where the last isomorphism is given by (\ref{iso-stable}), for $m$ big enough.
Hence:
$$\Hst(n,n)^0\simeq \kk[\mathfrak{S}_n]$$
and  the endomorphisms in the subPROP $\Hst'$  correspond to the endomorphisms in $\Hst$ in degree $0$.
\end{rem}

\begin{rem} \label{PROP-B-MCG}
The wheeled PROP structure on $\Hst$ comes from the wheeled endomorphism PROP and does not depend on the family of groups considered. Consequently, there are other families of groups for which we have  a wheeled PROP structure on the stable cohomology with coefficients given by $B_{l,q}=Hom_{\mathcal{V}}(\abe^{\otimes l}, \abe^{\otimes q} )$. For example, for the braid groups we have a wheeled PROP $\Hst^{B_\infty}$ and the group morphism $B_n \to Aut(\Z^{*n})$ induces a morphism of wheeled PROP $\Hst \to \Hst^{B_\infty}$. Similarly, for $\Sigma_{g,1}$ a connected and oriented surface of genus $g$ with $1$ boundary component and $\mathcal{M}_{g,1}$ its mapping class group,  we have a wheeled PROP $\Hst^{MCG_{\infty,1}}$ and the group morphism $\mathcal{M}_{g,1}\to Aut(\Z^{*2g})$ gives a morphism of wheeled PROP $\Hst\to {\Hst^{MCG_{\infty,1}}} $. The wheeled PROPs $\Hst^{B_\infty}$ and $\Hst^{MCG_{\infty,1}}$ have further structure. 
This will be developed elsewhere.
\end{rem}
\section{Cohomological classes} \label{Kawa-classes}
In  \cite{Kawa-Magnus} (see also \cite{Kawa-TMM}) the first author introduced cohomology classes that give non-zero morphisms in the PROP $\Hst$. In this section we show that these classes are obtained from the class $h_1$, recalled below, using the wheeled PROP structure on $\mathcal{H}$.

The $q$-th Johnson map induced by a Magnus expansion $\theta$ is  a map 
$$\tau^\theta_q: Aut(\Z^{*n}) \to  Hom_{\mathcal{V}}(\kk^n, (\kk^n)^{\otimes q+1} ).$$
By  \cite[Lemma 2.1]{Kawa-Magnus} $\tau^\theta_1$ is a $1$-cocycle and the cohomology class $h_1=[\tau^\theta_1] \in H^1(Aut(\Z^{*n}), 
Hom_{\mathcal{V}}(\kk^n, (\kk^n)^{\otimes 2}))$ does not depend on the choice of  Magnus expansion $\theta$.  For $n$ big enough, $h_1$ gives a non-zero element in $Hom_\Hst(2,1)$ in cohomological degree $1$. Using  \cite[(4.4)]{Kawa-Magnus} and the anticommutativity of the cup product we obtain that $\mathfrak{S}_2$ acts on $h_1$ by the signature.

By  \cite[(4.11)]{Kawa-Magnus} we have the relation in $Hom_\Hst(3,1)$
\begin{equation} \label{relation-Kawa}
h_1 \circ (h_1 \otimes 1) + h_1 \circ (1 \otimes h_1) =0
\end{equation}
where $\otimes$ is the horizontal composition in the PROP $\Hst$ and $\circ$ is the vertical composition in the PROP $\Hst$.

Let  $\mathcal{K}$ be the sub-wheeled PROP of $\mathcal{H}$ generated by the class $h_1$. 
\begin{prop} \label{Prop-Kawa}
For $p\in \N$, the classes $h_{p+1} \in \mathcal{K}(p+2,1)$ defined inductively by
$$h_{p+1}=h_1 \circ (h_p \otimes 1) $$
and $\overline{h_p} \in  \mathcal{K}(p, 0)$ defined by  
$$\overline{h_p} = \xi_1^1(h_p)$$
are the cohomological classes introduced in  \cite{Kawa-Magnus}, in the stable range.
\end{prop}
\begin{proof}
For $p\geq 2$, using the cup product we obtain classes 
$$(h_1)^{\cup p} \in H^p(Aut(\Z^{*n}), 
(Hom_{\mathcal{V}}(\kk^n, (\kk^n)^{\otimes 2})^{\otimes p}) \simeq H^p(Aut(\Z^{*n}), 
Hom_{\mathcal{V}}((\kk^n)^{\otimes p}, (\kk^n)^{\otimes 2p}) $$
where the isomorphism is induced by  $\lambda$ given in (\ref{lambda-iso}).
 In the stable range, we obtain $(h_1)^{\cup p} \in \Hst(2p,p)$ and the previous construction corresponds to the horizontal composition in the PROP $\Hst$ introduced in Definition \ref{def-H}.

Consider the maps 
$$\varsigma_p: ((\kk^n)^* \otimes (\kk^n)^{\otimes 2})^{\otimes p} \simeq Hom_\V(\kk^n,(\kk^n)^{\otimes 2})^{\otimes p} \to Hom_\V(\kk^n,(\kk^n)^{\otimes p+1}) \simeq (\kk^n)^* \otimes (\kk^n)^{\otimes p+1}$$
given by
$$\varsigma_p(u_1 \otimes u_2 \otimes \ldots\otimes u_p):=(u_1\otimes 1_{(\kk^n)^{\otimes p-1})})\circ (u_2 \otimes 1_{(\kk^n)^{\otimes p-2})}) \circ \ldots \circ (u_{p-1} \otimes 1_{\kk^n}) \circ u_p$$
where $u_i \in Hom_\V(\kk^n,(\kk^n)^{\otimes 2})$ for $1\leq i \leq p$  (see \cite[(4.8)]{Kawa-Magnus}). 

The cohomological classes $h_p \in H^p( Aut(\Z^{*n}), Hom_{\mathcal{V}}(\kk^n, (\kk^n)^{\otimes p+1}))$ are defined in \cite[Theorem 4.1]{Kawa-Magnus} by
$$h_p=H^p( Aut(\Z^{*n}), \varsigma_p) (h_1^{\cup p}).$$
 Note that
$$\varsigma_p(u_1 \otimes u_2 \otimes \ldots\otimes u_p)=(\varsigma_{p-1}(u_1 \otimes u_2 \otimes \ldots\otimes u_{p-1}) \otimes 1_{\kk^n}) \circ u_p$$

It follows  that, for $n$ big enough, the classes $h_p$ can be defined recursively by: 
$$h_{p+1}=h_1 \circ (h_p \otimes 1) \in \Hst(p+2,1).$$

Consider the map 
$$\varphi^1_1: (\kk^n)^* \otimes (\kk^n)^{\otimes p+1} \to (\kk^n)^{\otimes p}$$ 
introduced in Example \ref{ex-wP}. The reduced class $\overline{h_p} \in H^p( Aut(\Z^{*n}), (\kk^n)^{\otimes p})$ is defined, in  \cite[(4.7)]{Kawa-Magnus}, from the class $h_p$ by
$$\overline{h_p}=H^p( Aut(\Z^{*n}), \varphi^1_1)(h_p).$$ 
In the stable range, this corresponds to considering the contraction $ \xi_1^1: \Hst(p+1,1) \to \Hst(p, 0)$ introduced in Proposition \ref{H-wP}, so we have $\overline{h_p} = \xi_1^1(h_p)$.

\end{proof}

\begin{rem} \label{xi^2}
By the bi-equivariance condition for the contraction map we have:
$$\xi_1^2(h_1)=-\overline{h_1}.$$
\end{rem}

\begin{rem}\label{Rem-M-KM}
The wheeled PROP $\Hst^{MCG_{\infty,1}}$ evoked in Remark \ref{PROP-B-MCG}
is related to the graph description of the (twisted) Mumford-Morita-Miller classes by Morita and the first author \cite{Mo95, KM01}.
 Morita \cite{Mo93} extended the first Johnson homomorphism of the Torelli group $\mathcal{I}_{g,1}$ to 
a twisted cohomology class $\tilde k \in H^1(\mathcal{M}_{g,1}; 
\tfrac12\Lambda^3\abe(\pi_1(\Sigma_{g,1})))$. 
The class $h_1$ restricts to $\tilde k$ on the mapping class group 
$\mathcal{M}_{g,1}$. 
Morita \cite{Mo95} constructed cohomology classes of  
the mapping class group with trivial coefficients $\mathbb{Q}$
by contracting a power of the class 
$\tilde k$ in terms of trivalent graphs. 
More precisely, any trivalent graph $\Gamma$ with $2n$ vertices defines an $Sp(\abe(\pi_1(\Sigma_{g,1})))$-invariant linear map $\alpha_{\Gamma}: \Lambda^{2n}(\Lambda^3\abe(\pi_1(\Sigma_{g,1}))\otimes\mathbb{Q}) \to \mathbb{Q}$ by using the intersection pairing on $\abe(\pi_1(\Sigma_{g,1}))$. Then we obtain a cohomology class 
${\alpha_\Gamma}_*({\tilde k}^{2n}) \in H^{2n}(\mathcal{M}_{g,1}; \mathbb{Q})$. Morita and the first author 
proved all these classes are polynomials in the Mumford-Morita-Miller classes \cite{KM96}, and generalized his construction to 
all finite graphs and twisted Mumford-Morita-Miller classes 
\cite{KM01}. Here any graph with $n$ univalent vertices 
defines a cohomology class of $\mathcal{M}_{g,1}$ with coefficients 
in $\Lambda^n\abe(\pi_1(\Sigma_{g,1}))\otimes\mathbb{Q}$, which is proved to be a polynomial of twisted Mumford-Morita-Miller classes. 

\end{rem}
\part{Functor cohomology in $\F(\gr; \kk)$} \label{reso-abe}
The aim of this part is to introduce the wheeled PROP $\Ext$ given by $Ext$-groups in the functor category $\F(\gr)$.
\section{Projective resolution of the abelianization functor}
The bar resolution gives rise to a projective resolution of the abelianization functor of the form (see \cite[Proposition 5.1]{JP} or \cite[Section 1]{Vespa}) :
\begin{equation*}
\ldots \to P_{n+1} \to P_n \to \ldots \to P_2 \to P_1 
\end{equation*}
Considering the normalized bar resolution we obtain a variant of the previous resolution of the form:
\begin{equation} \label{resolution}
\ldots \to \bar{P}^{ \otimes n+1} \xrightarrow{d_{n}}  \bar{P}^{\otimes n} \xrightarrow{d_{n-1}}  \ldots \xrightarrow{d_2}  \bar{P}^{\otimes 2} \xrightarrow{d_1} \bar{P} 
\end{equation}
where the map $\pi: \bar{P} \to \abe$ corresponds to $1\in \Z$ via the  isomorphism $Hom(\bar{P}, \abe)\simeq \Z$ obtained using Proposition \ref{Prop-barP}.

\section{The PROP $\Ext^0$} \label{E^0}
The aim of this section is to describe the structure of the following graded PROP which can be viewed as an $Ext$-version of the endomorphism PROP:
\begin{defn} \label{PROPExt0}
The graded linear PROP $\Ext^0$ is defined by the $(\mathfrak{S}_q, \mathfrak{S}_l)$-graded bimodules 
$${\Ext^0}(q,l)=Ext^*_{\F(\gr; \kk)}(\abe^{\otimes l}, \abe^{\otimes q})$$
where the action of the symmetric group $\mathfrak{S}_q$ (resp.  $\mathfrak{S}_l$) is given by the permutations of the copies of $\abe$ in the first variable (resp. in the second variable).

The horizontal composition ${\Ext^0}(q_1,l_1) \otimes {\Ext^0}(q_2,l_2) \to {\Ext^0}(q_1+q_2,l_1+l_2)$  is given by the exterior product and the vertical composition $\Ext^0(q,l) \otimes \Ext^0(l,m) \to \Ext^0(q,m)$ is given by the Yoneda product.
\end{defn}

\begin{rem}
We warn the reader that $Hom_{\F(\gr)}(\abe^{\otimes l}, \abe^{\otimes q} )$ should not be confused with $Hom_{\mathcal{V}}(\abe^{\otimes l}, \abe^{\otimes q} )$ introduced at the end of Section \ref{def-stab}. 

\end{rem}
In \cite{Vespa}, the second author obtained the following results:
\begin{thm} \cite[Theorem 2.3]{Vespa} \label{Ext}
For $l,q \in \N$, we have an isomorphism:
$$Ext^{*}_{\F(\gr; \kk)}( \abe^{\otimes l},  \abe^{\otimes q}) \simeq \left\lbrace\begin{array}{ll}
 \kk[Surj(\textbf{q},\textbf{l})] & \text{if } *=q-l\\
 0 & \text{otherwise}
 \end{array}
 \right.$$
  where $Surj(\textbf{q},\textbf{l})$ is the set of surjections from $\textbf{q}$ to $\textbf{l}$.
  \end{thm}
  
  \begin{thm} \cite[Proposition 2.5]{Vespa} \label{Ext-action}
  The symmetric groups $\mathfrak{S}_q$ and $\mathfrak{S}_l$ act on $Ext^{q-l}_{\F(\gr; \kk)}( \abe^{\otimes l},  \abe^{\otimes q}) \simeq \kk[Surj(\textbf{q},\textbf{l})]$ in the following way: for $\sigma \in \mathfrak{S}_q$, $\tau_{a,b} \in \mathfrak{S}_l$ the transposition of $a$ and $b$ where $a,b \in \{1, \ldots, l\}$ and $f \in Surj(\textbf{q},\textbf{l})$
$$[f].\sigma= \underset{1\leq i \leq l}{\prod}\epsilon(\overline{\sigma_{\mid (f \circ \sigma)^{-1}(i)}})[f\circ \sigma]$$
where $\sigma_{\mid (f \circ \sigma)^{-1}(i)}:(f \circ \sigma)^{-1}(i)\to \sigma((f \circ \sigma)^{-1}(i))$ 
$$\tau_{a,b}.[f]=(-1)^{(|f^{-1}(a)|-1)(|f^{-1}(b)|-1)} [\tau_{a,b} \circ f].$$
\end{thm}

\begin{prop} \cite[Proposition 3.1]{Vespa} \label{products}
The external product
$$e: Ext^{m-l}_{\F(\gr; \kk)}( \abe^{\otimes l},  \abe^{\otimes m}) \otimes Ext^{n-p}_{\F(\gr; \kk)}( \abe^{\otimes p}, \abe^{\otimes n}) \to Ext^{m+n-l-p}_{\F(\gr; \kk)}(\abe^{\otimes l+p}, \abe^{\otimes m+n})$$
is induced by the disjoint union of sets via the isomorphism of  Theorem \ref{Ext}.

\end{prop}

For $c_{m-l}$ (resp. $c_{n-p}$) a cocycle representing a generator of $Ext^{m-l}_{\F(\gr; \kk)}( \abe^{\otimes l},  \abe^{\otimes m}) $ (resp. $Ext^{n-p}_{\F(\gr; \kk)}( \abe^{\otimes p}, \abe^{\otimes n})) $, we will denote $e([c_{m-l}], [c_{n-p}])$  by $[c_{m-l}] \otimes [c_{n-p}]$.

Note that in the description of the Yoneda product in terms of surjection given in \cite[Proposition 3.1]{Vespa} the signs are not correct. One of the aim of Sections \ref{explicit-classes} and \ref{operad-P0} is to give a corrigendum of this statement.

\subsection{Explicit classes in $Ext^{n-1}_{\F(\gr; \kk)}( \abe,  \abe^{\otimes n})$} \label{explicit-classes}
The aim of this section is to construct explicit cocycles representing the generators  in $Ext^{n-1}_{\F(\gr; \kk)}( \abe,  \abe^{\otimes n})$ and to study their behaviour via  the Yoneda product:
$$\mathcal{Y}: Ext^{n-1}_{\F(\gr; \kk)}(\abe, \abe^{\otimes n}) \otimes Ext^{1}_{\F(\gr; \kk)}(\abe^{\otimes n}, \abe^{\otimes n+1}) \to Ext^{n}_{\F(\gr; \kk)}(\abe, \abe^{\otimes n+1})$$

We begin by introducing explicit classes  in $Ext^{n-1}_{\F(\gr; \kk)}( \abe,  \abe^{\otimes n}) \simeq \kk$.
\begin{prop} \label{cocycle}
For $n \in \N \setminus \{0\}$, the morphism $\pi^{\otimes n}: \bar{P}^{\otimes n} \to \abe^{\otimes n}$ is a cocycle representing a generator of $Ext^{n-1}_{\F(\gr; \kk)}( \abe,  \abe^{\otimes n}) \simeq \kk$.
\end{prop}
\begin{proof}
Using the normalized bar resolution (\ref{resolution}), $Ext^{n-1}_{\F(\gr; \kk)}( \abe,  \abe^{\otimes n})$ is the homology of the complex:
$$\ldots \to Hom_{\F(\gr; \kk)}(\bar{P}^{\otimes n-1}, \abe^{\otimes n}) \xrightarrow{d} Hom_{\F(\gr; \kk)}(\bar{P}^{\otimes n}, \abe^{\otimes n}) \xrightarrow{d}  Hom_{\F(\gr; \kk)}(\bar{P}^{\otimes n+1}, \abe^{\otimes n}) \to \ldots$$
By Corollary \ref{Cor-barP},  $Hom_{\F(\gr; \kk)}(\bar{P}^{\otimes n+1}, \abe^{\otimes n}) =0$ since $\abe^{\otimes n}$ is a polynomial functor of degree $n$. It follows that $d(\pi^{\otimes n})=0$, 

 moreover, the morphism $\pi^{\otimes n}$ represents $[Id_n]$ via the isomorphism
$$Hom_{\F(\gr; \kk)}(\bar{P}^{\otimes n}, \abe^{\otimes n}) \simeq cr_{n} \abe^{\otimes n}(\Z, \ldots, \Z)\simeq \kk[\mathfrak{S}_n]$$
using the external product $Hom(\bar{P}, \abe)\otimes \ldots \otimes Hom(\bar{P}, \abe) \to Hom(\bar{P}^{\otimes n}, \abe^{\otimes n})$.
We deduce from the previous complex, an exact sequence of $\mathfrak{S}_n$-modules:
$$ Hom_{\F(\gr; \kk)}(\bar{P}^{\otimes n-1}, \abe^{\otimes n}) \to \kk[\mathfrak{S}_n] \to  Ext^{n-1}_{\F(\gr; \kk)}( \abe,  \abe^{\otimes n}) \to 0$$
It follows that $[Id_n]$ gives a generator of the $\mathfrak{S}_n$-module $Ext^{n-1}_{\F(\gr; \kk)}( \abe,  \abe^{\otimes n}) \simeq \kk$.
\end{proof}

In the next Proposition we introduce particular classes in $Ext^{1}_{\F(\gr; \kk)}(\abe^{\otimes n}, \abe^{\otimes n+1})$.
Let $Q^\bullet_n\to  \abe^{\otimes n}$ be a projective resolution of $ \abe^{\otimes n}$ and consider the resolution $\bar{P}^{\otimes \bullet +1} \otimes \abe^{\otimes n-1} \to \abe^{\otimes n}$ obtained by tensoring the complex (\ref{resolution}) with $\abe^{\otimes n-1} $. By standard homological algebra (see \cite[Comparison Theorem 2.2.6]{Weibel}) there is a chain map $\alpha^\bullet: Q^\bullet_n \to \bar{P}^{\otimes \bullet +1} \otimes \abe^{\otimes n-1}$ lifting $Id_{\abe^{\otimes n}}: \abe^{\otimes n} \to \abe^{\otimes n}$ which is unique up to chain homotopy equivalence.

\begin{lm}
The map 
$$(\pi^{\otimes 2} \otimes Id_{\abe^{\otimes n-1}}) \circ \alpha^1: Q^1_n \to \abe^{\otimes n+1}$$
represents the class of $Ext^{1}_{\F(\gr; \kk)}(\abe^{\otimes n}, \abe^{\otimes n+1})$ corresponding to the exterior product $[\pi^{\otimes 2}] \otimes [Id_{\abe^{\otimes n-1} }]$.
\end{lm}
\begin{proof}
This is a direct consequence of the definition of the exterior product of classes.
\end{proof}
\begin{lm}
The functor $Im(d_{n-1})$ has projective resolution 
\begin{equation} \label{resolution2}
\bar{P}^{ \otimes \bullet + n}: \qquad \ldots \to \bar{P}^{ \otimes n+1} \xrightarrow{d_{n}}  \bar{P}^{\otimes n} 
\end{equation}
given by truncating (\ref{resolution}). Moreover the map $\pi^{\otimes n}: \bar{P}^{\otimes n} \to \abe^{\otimes n}$ factorizes through $Im(d_{n-1})$ giving rise to a morphism $\overline{\pi^{\otimes n}}: Im(d_{n-1}) \to \abe^{\otimes n}$ 
\end{lm}
\begin{proof}
By the projective resolution (\ref{resolution2}) we obtain the complex
$$ 
\ldots \leftarrow Hom(\bar{P}^{ \otimes n+1}, \abe^{\otimes n}) \xleftarrow{d_{n}^*}  Hom(\bar{P}^{\otimes n}, \abe^{\otimes n}) 
$$
computing $Ext^i_{\F(\gr)}(Im(d_{n-1}), \abe^{\otimes n})$. The map $\pi^{\otimes n}: \bar{P}^{\otimes n} \to \abe^{\otimes n}$ satisfies $\pi^{\otimes n} \circ d_n=0$ so it represents a cocycle in $Hom_{\F(\gr)}(Im(d_{n-1}), \abe^{\otimes n})$. We deduce that $\pi^{\otimes n}$ factorizes through $Im(d_{n-1})$ giving rise to a morphism $\overline{\pi^{\otimes n}}: Im(d_{n-1}) \to \abe^{\otimes n}$ 
\end{proof}
\begin{lm} \label{lifting}
We have a morphism of exact chain complexes 
$$\xymatrix{
\ldots \ar[r] \ar[d]&\bar{P}^{\otimes k+1}\otimes  \bar{P}^{\otimes n-1} \ar[r]^-{d_{n+k-1}} \ar[d]^{Id \otimes \pi^{\otimes n-1}}&\ldots \ar[r] \ar[d]& \bar{P}^{\otimes 2}\otimes  \bar{P}^{\otimes n-1} \ar[r]^-{d_n} \ar[d]^{Id \otimes \pi^{\otimes n-1}}&\bar{P} \otimes \bar{P}^{\otimes n-1} \ar[d]^{Id \otimes \pi^{\otimes n-1}}\ar[r]^-{d_{n-1}} &  Im(d_{n-1}) \ar[r] \ar[d]^{\overline{\pi^{\otimes n}}}& 0\\
\ldots \ar[r] &\bar{P}^{\otimes k+1} \otimes \abe^{\otimes n-1} \ar[r]_{d_{k} \otimes Id}  &\ldots \ar[r] &\bar{P}^{\otimes 2} \otimes \abe^{\otimes n-1} \ar[r]_{d_1 \otimes Id}  & \bar{P}\otimes \abe^{\otimes n-1}\ar[r]_-{\pi \otimes Id}& \abe^{\otimes n} \ar[r] & 0
}$$
\end{lm}
\begin{proof}
The square on the right commutes since 
$$\overline{\pi^{\otimes n}}\circ d_{n-1}=\pi^{\otimes n}=(\pi \otimes Id) \circ (Id \otimes \pi^{\otimes n-1})$$
For $k \in \mathbb{N}$, a direct computation using the differential in the reduced bar resolution gives the commutativity of the diagram
$$\xymatrix{
\bar{P}^{\otimes k+2}\otimes  \bar{P}^{\otimes n-1} \ar[r]^-{d_{n+k}} \ar[d]^{Id \otimes \pi^{\otimes n-1}}&\bar{P}^{\otimes k+1}\otimes  \bar{P}^{\otimes n-1} \ar[d]^{Id \otimes \pi^{\otimes n-1}}\\
\bar{P}^{\otimes k+2} \otimes \abe^{\otimes n-1} \ar[r]_{d_{k+1} \otimes Id}  &\bar{P}^{\otimes k+1} \otimes \abe^{\otimes n-1}.
}$$

\end{proof}
\begin{rem}
The readers's attention is drawn to the fact that the following diagram is only commutative \textit{up to a sign}
$$\xymatrix{
\bar{P}^{\otimes n-1}\otimes \bar{P}^{\otimes k+2}  \ar[r]^-{d_{n+k}} \ar[d]^{\pi^{\otimes n-1} \otimes Id}& \bar{P}^{\otimes n-1} \otimes \bar{P}^{\otimes k+1} \ar[d]^{  \pi^{\otimes n-1} \otimes Id}\\
\abe^{\otimes n-1} \otimes \bar{P}^{\otimes k+2} \ar[r]_{ Id\otimes d_k}  &\abe^{\otimes n-1}\otimes \bar{P}^{\otimes k+1}.
}$$
\end{rem}
\begin{prop} \label{Yoneda}
For $n \geq 2$
$$\mathcal{Y}([\pi^{\otimes n}], [\pi^{\otimes 2}] \otimes [Id_{\abe^{\otimes n-1} }])=[\pi^{\otimes n+1}]$$
\end{prop}
\begin{proof}
By \cite[Comparison Theorem 2.2.6]{Weibel} there is a chain map $\beta^\bullet$ lifting $\overline{\pi^{\otimes n}}: Im(d_{n-1}) \to \abe^{\otimes n}$. We obtain the following  commutative diagram
$$\xymatrix{
\ldots \ar[r] \ar[d]& \bar{P}^{\otimes n+1} \ar[r]^-{d_n} \ar[d]^{\beta^1}& \bar{P}^{\otimes n} \ar[d]^{\beta^0} \ar[r]^-{d_{n-1}} &  Im(d_{n-1}) \ar[r] \ar[d]^{\overline{\pi^{\otimes n}}}& 0\\
\ldots \ar[r] \ar[d]&Q_n^1  \ar[r]  \ar[d]^{\alpha^1}& Q_n^0 \ar[r]\ar[d]^{\alpha^0}& \abe^{\otimes n} \ar[r] \ar[d]^{Id}& 0\\
\ldots \ar[r] &\bar{P}^{\otimes 2} \otimes \abe^{\otimes n-1} \ar[r]  & \bar{P}^{\otimes 1} \otimes \abe^{\otimes n-1}\ar[r]& \abe^{\otimes n} \ar[r] & 0
}$$
By the construction of the Yoneda product,  $(\pi^{\otimes 2} \otimes Id_{\abe^{\otimes n-i}}) \circ \alpha^1 \circ \beta^1$ is a cocycle representing $\mathcal{Y}([\pi^{\otimes n}], [(\pi^{\otimes 2} \otimes Id_{\abe^{\otimes n-i}}) \circ \alpha^1])$.

Since the chain map lifting $\overline{\pi^{\otimes n}}: Im(d_{n-1}) \to \abe^{\otimes n}$ is unique up to chain homotopy equivalence, using Lemma \ref{lifting} we have
$$[(\pi^{\otimes 2} \otimes Id_{\abe^{\otimes n-i}}) \circ \alpha^1 \circ \beta^1]=[(\pi^{\otimes 2} \otimes Id_{\abe^{\otimes n-i}}) \circ (Id \otimes \pi^{\otimes n-1})] =[\pi^{\otimes n+1}].$$
\end{proof}


\subsection{The operad $\mathcal{P}_0$} \label{operad-P0}
In \cite[Proposition 3.5]{Vespa} the second author proved that the graded PROP $\Ext^0$ is freely-generated by its underlying operad $\mathcal{P}_0$.  Using Theorems \ref{Ext} and \ref{Ext-action},   $\mathcal{P}_0$  is the graded operad such that $\mathcal{P}_0(0)=0$ and for $k>0$, $\mathcal{P}_0(k)$ is the sign representation of $\mathfrak{S}_k$ placed in cohomological degree $k-1$ and $0$ in other degrees. The aim of this section is to give an explicit description of this operad, in particular to describe the composition which is induced by the Yoneda product.

\begin{defn} \label{operade-Q}
The operad $\mathcal{Q}$ is the  quadratic graded operad  generated by one antisymmetric binary operation $\mu$ in degree $1$ subject to the quadratic relation:
$$\mu \underset{1}{\circ} \mu =-\mu \underset{2}{\circ} \mu.$$
Pictorially, we have

\[
	\vcenter{\hbox{\begin{tikzpicture}[baseline=1.8ex,scale=0.3]
	\draw (2,0) -- (2,2);
	\draw (3,4) -- (2,2);
	\draw (1,4) -- (2,2);
	\draw (1,4) -- (2,6);
	\draw (1,4) -- (0,6);
	\draw[fill=white] (2,2) circle (8pt);
	\draw[fill=white] (1,4) circle (8pt);
		\end{tikzpicture}}}
		\qquad =\qquad 
		-\vcenter{\hbox{\begin{tikzpicture}[baseline=1.8ex,scale=0.3]
	\draw (2,0) -- (2,2);
	\draw (3,4) -- (2,2);
	\draw (1,4) -- (2,2);
	\draw (3,4) -- (4,6);
	\draw (3,4) -- (2,6);
	\draw[fill=white] (2,2) circle (8pt);
	\draw[fill=white] (3,4) circle (8pt);
		\end{tikzpicture}}}
				\]
\end{defn}
\begin{prop} \label{operade-gr}
The operad $\mathcal{P}_0$ is isomorphic to  $\mathcal{Q}$.
\end{prop}
\begin{proof}
We show that there is a morphism of operads $f: \mathcal{Q} \to \mathcal{P}_0$ given on the generator $\mu$ by $f_2(\mu)=[\pi^{\otimes 2}]$ where $[\pi^{\otimes 2}]$ is a  generator, in degree 1,  of $\Ext^0(2,1)=\mathcal{P}_0(2)$ defined in Proposition \ref{cocycle} which is antisymmetric by Theorem \ref{Ext-action}.

Before proving that $[\pi^{\otimes 2}]$ satisfies the quadratic relation satisfied by $\mu$, note that the partial composition operations in $\mathcal{P}_0$ are given by the restriction of the categorical composition induced  by the Yoneda product
$$\mathcal{Y}: \mathcal{P}_0(n)\underset{\mathfrak{S}_n}{\otimes} \Ext^0(n+1,n) \to \mathcal{P}_0(n+1)$$
More explicitly the partial composition $\underset{i}{\circ}$ is obtained by restricting to the inclusion of $\kk$-modules:
$$\xi^n_i: \mathcal{P}_0(2)\simeq \mathcal{P}_0(1)^{\otimes i-1} \otimes \mathcal{P}_0(2) \otimes \mathcal{P}_0(1)^{\otimes n-i} \hookrightarrow \Ext^0(n+1,n)$$
given by the external product.

For $c_2 \in \mathfrak{S}_2$ and $c_3\in \mathfrak{S}_3$ the cyclic permutations given by $i \mapsto i+1$, by Theorem \ref{Ext-action} we have in $Ext^{1}_{\F(\gr; \kk)}(\abe^{\otimes 2}, \abe^{\otimes 3})$:
$$c_2.([Id_{\abe}] \otimes [\pi^{\otimes 2}]).c_3=[\pi^{\otimes 2}] \otimes [Id_{\abe}]$$
So:
$$\mathcal{Y}([\pi^{\otimes 2}], [\pi^{\otimes 2}] \otimes [Id_{\abe}])=\mathcal{Y}([\pi^{\otimes 2}], c_2.([Id_{\abe}] \otimes [\pi^{\otimes 2}]).c_3)=\mathcal{Y}([\pi^{\otimes 2}]. c_2, ([Id_{\abe}] \otimes [\pi^{\otimes 2}]))
.c_3$$
$$=\mathcal{Y}([\pi^{\otimes 2}](-1), ([Id_{\abe}] \otimes [\pi^{\otimes 2}]))=-\mathcal{Y}([\pi^{\otimes 2}], ([Id_{\abe}] \otimes [\pi^{\otimes 2}]))$$
using that $\mathcal{P}_0(k)$ is the sign representation in degree $k-1$. We deduce that $[\pi^{\otimes 2}]$ satisfies the quadratic relation.

The fact that the operad $\mathcal{P}_0$ is binary follows from Proposition \ref{Yoneda} and the fact that $\mathcal{P}_0(n+1)$ is $\kk$ concentrated in degree $n$. 

We deduce that the morphism of operads $f: \mathcal{Q} \to \mathcal{P}_0$ is an isomorphism.

\end{proof}

In Proposition \ref{operad-sus} we give a more conceptual description of the operad $\mathcal{P}_0$ using the notion of \textit{operadic suspension},  recalled in the next paragraphs following \cite[Section 7.2.2]{LV}.

Let $\mathcal{S}$ be the underlying operad of the endomorphism PROP (see Example \ref{ex-P}) associated with the graded vector space $s\kk$ (i.e. the graded vector space concentrated in homological  degree one and such that $(s\kk)_1=\kk$). Explicitly, as a representation of $\mathfrak{S}_n$, $\mathcal{S}(n)=Hom_{gr\V}((s\kk)^{\otimes n}, s\kk)$ is the signature representation concentrated in homological degree $-n+1$.

For $\mathcal{P}$ and  $\mathcal{Q}$ two operads, the Hadamard tensor product $\mathcal{P} \underset{H}{\otimes} \mathcal{Q}$ is an operad such that
$$(\mathcal{P} \underset{H}{\otimes} \mathcal{Q})(n)=\mathcal{P}(n) \underset{H}{\otimes} \mathcal{Q}(n)$$
where the action of $\mathfrak{S}_n$ is the diagonal action. The unit of the Hadamard product is the operad $u \mathcal{C}om$ of \textit{unital} commutative algebras. In particular we have $u \mathcal{C}om(0) =\kk$.

For $\mathcal{P}$ an operad, the \textit{operadic suspension} of $\mathcal{P}$ is the operad $\mathcal{S} \underset{H}{\otimes} \mathcal{P}$.

The operad $\mathcal{P}_0$ being defined using $Ext$-groups, it is naturally graded with cohomological degree. So $\mathcal{P}_0(n)$ is a graded $\kk$-module concentrated in \textit{cohomological} degree $n-1$ and so in \textit{homological} degree $1-n$. 

Let $\mathcal{C}om$ be the operad of  \textit{non-unital} commutative algebras (thus $\mathcal{C}om(0) =0$).
 In the next Proposition we consider $\mathcal{P}_0$ with its homological degree.

\begin{prop}\label{operad-sus} 
The operad $\mathcal{P}_0$ is the operadic suspension of the operad $\mathcal{C}om$. In other words, we have an isomorphism of operad
$$\mathcal{P}_0 \simeq \mathcal{S} \underset{H}{\otimes} \mathcal{C}om.$$
\end{prop}
\begin{proof}
Recall that $\mathcal{P}_0(0)=0$ and for $n>0$, $\mathcal{P}_0(n)$ is the sign representation of $\mathfrak{S}_n$ placed in homological degree $1-n$ and $0$ in other degrees.\\
The underlying $\mathfrak{S}_n$-modules are isomorphic since $\mathcal{C}om(0)=0$ and for $n\geq 1$, $\mathcal{P}_0(n) \simeq  \mathcal{S}(n)$.

By Proposition \ref{operade-gr}, $\mathcal{P}_0$ is a quadratic operad generated by $\mu$ such that 
$\mu \underset{1}{\circ} \mu =-\mu \underset{2}{\circ} \mu.$ Define $\mathcal{P}_0(2) \to \mathcal{S}(2)$ sending $\mu$ to the generator $\nu$ of $\mathcal{S}(2)$. By a similar argument as in the proof of Proposition  \ref{operade-gr} for the endomorphism PROP associated with the graded vector space $s\kk$, one proves that $\nu$ satisfies the relation $\nu \underset{1}{\circ} \nu =-\nu \underset{2}{\circ} \nu$. This implies the isomorphism in the statement.
\end{proof}

\begin{rem}
We denote by $s_i^n \in Surj(n,n-1)$ the unique surjection preserving the natural order and such that $s_i^n(i)=s_i^n (i+1)=i$. The Yoneda product gives, via the isomorphism given in Theorem \ref{Ext}, a map
$$Y: \kk[Surj(m,l)] \otimes \kk[Surj(n,m)] \to \kk[Surj(n,l)]$$
For $m=2$, $l=1$ and $n=3$, the quadratic relation in the operad $\mathcal{P}_0$ corresponds to: $Y([s_1^2] \otimes [s_2^3])=-Y([s^2_1] \otimes [s_1^3])$ showing that the signs given in  \cite[Proposition 3.1]{Vespa} are not correct.
\end{rem}

The following maps will be used in Proposition \ref{wheeled-operad} in order to define the contraction maps in the wheeled operad $\mathcal{P}$.
\begin{defn} \label{contr-def}
For $1 \leq i\leq n$, let $\xi^i: \mathcal{P}_0(n) \to \mathcal{P}_0(n-1)$ be the morphism of graded vector spaces of degree $-1$ given by:
$$\xi^i([\pi^{\otimes n}])=(-1)^i[\pi^{\otimes n-1}].$$ 
\end{defn}

\begin{prop} \label{contr-equ}
For $1 \leq i\leq n$, the contraction map $\xi^i :\mathcal{P}_0(n) \to \mathcal{P}_0(n-1)$ is equivariant i.e. for $\sigma \in \mathfrak{S}_n$, the following diagram is commutative 
$$\xymatrix{
\mathcal{P}_0(n) \ar[r]^-{\xi^i}  \ar[d]_{\sigma} & \mathcal{P}_0(n-1)  \ar[d]^{{\sigma}^{(i)}} \\
\mathcal{P}_0(n) \ar[r]_{\xi^{\sigma(i)}} &\mathcal{P}_0(n-1)  
}
$$
where $\sigma^{(i)} \in \mathfrak{S}_{n-1}$ is the permutation $\sigma: \textbf{n}\setminus\{i\} \to \textbf{n}\setminus\{ \sigma(i)\}$  considered as reindexed.
\end{prop}

\begin{proof}
Since $\mathfrak{S}_n$ acts on $\mathcal{P}_0(n)$ by the signature, we require to prove that
$$ \epsilon(\sigma) (-1)^{\sigma(i)}=\epsilon({\sigma}^{(i)})(-1)^i.$$
Let $\alpha \in \mathfrak{S}_n$ be the permutation given by
$$\alpha=(1,2) \circ (2,3) \circ \ldots \circ (\sigma(i)-1, \sigma(i)) \circ \sigma \circ (i-1, i) \circ \ldots \circ (1,2)$$
where $(l,l+1)$ is the permutation exchanging $l$ and $l+1$. We have
$$\epsilon(\alpha)=(-1)^{\sigma(i)-1} \epsilon(\sigma) (-1)^{i-1}$$
and for ${\alpha}^{(1)}$ the permutation  ${\alpha}:  \textbf{n}\setminus\{1\} \to \textbf{n}\setminus\{ \alpha(1)=1\}$ considered as reindexed. We have ${\alpha}^{(1)}={\sigma}^{(i)}$ and $\alpha=Id_{\{1\}} \oplus {\alpha}^{(1)}$. Hence
$$\epsilon({\sigma}^{(i)})=\epsilon({\alpha}^{(1)})=\epsilon(\alpha)=(-1)^{\sigma(i)-1} \epsilon(\sigma) (-1)^{i-1}.$$
\end{proof}


\section{The PROP $\Ext$} \label{PROP-E}

In the rest of the paper $\kk$ is a field of characteristic zero. The previous condition on $\kk$  allows us to use the computation of $Ext^{*}_{\F(\gr)}(\Lambda^j  \abe,  \abe^{\otimes q}) $ given in \cite{Vespa} (see Proposition \ref{Ext2} below) 
obtained taking the coinvariants by the action of the symmetric groups, twisted by the signature, in the result of Theorem \ref{Ext}.

The aim of this section is to describe the structure of the  graded PROP $\Ext$ introduced in Definition \ref{PROPExt}. 
We will prove in Theorem \ref{PROP-E-explicite} that the PROP $\Ext$ is a wheeled PROP.
The PROP $\Ext$ extends the PROP $\Ext^0$ in the sense that $\Ext^0$ is a subPROP of $\Ext$ (see Remark \ref{rem-subPROP}).

Since the exterior powers intervene in the PROP $\Ext$ we need the following result:
\begin{prop}\cite[Theorem 4.2]{Vespa} \label{Ext2}
For $\kk$ a field of characteristic $0$ and $n,m \in \mathbb{N}$, we have isomorphisms
$$Ext^{*}_{\F(\gr)}(\Lambda^j  \abe,  \abe^{\otimes q}) \simeq \left\lbrace\begin{array}{ll}
 \kk^{S(q,j)} & \text{if } *=q-j\\
 0 & \text{otherwise}
 \end{array}
 \right.$$
 where $S(q,j)$ denotes the number of ways to partition a set of $q$ elements into $j$ non-empty subsets
 
 $$Ext^{*}_{\F(\gr)}(\Lambda^n \abe, \Lambda^m  \abe) \simeq \left\lbrace\begin{array}{ll}
\kk^{\rho(m,n)} & \text{if } *=m-n\\
 0 & \text{otherwise}
 \end{array}
 \right.$$
 where $\rho(m,n)$ denotes the number of partitions of $m$ into $n$  parts.
\end{prop}

Since $Hom_{\F(\gr)}(\Lambda^j \abe, \Lambda^j \abe) \simeq \kk$, the external product gives a morphism
\begin{equation} \label{E}
Ext^*_{\F(\gr)}(\abe^{\otimes m} \otimes \Lambda^i \abe, \abe^{\otimes l})  \xrightarrow{E}Ext^*_{\F(\gr)}(\abe^{\otimes m} \otimes \Lambda^i  \abe \otimes \Lambda^j  \abe, \abe^{\otimes l} \otimes \Lambda^j  \abe).
\end{equation}

Recall that for $V^\bullet$ a cohomologically graded module, for $i \in \N$ the $i$-th desuspension of $V^\bullet$ is the graded module $s^{-i}V^\bullet$ such that $s^{-i}V^n=V^{n-i}$.
\begin{defn} \label{PROPExt}
The PROP $\Ext$ is defined by the graded $(\mathfrak{S}_q, \mathfrak{S}_l)$-bimodules 
$$\Ext(q,l)={\underset{j\in \N}{\bigoplus}}s^{-j}Ext^*_{\F(\gr; \kk)}(\abe^{\otimes l} \otimes \Lambda^j  \abe, \abe^{\otimes q})$$
where the action of the symmetric group $\mathfrak{S}_q$ (resp.  $\mathfrak{S}_l$) is given by place permutation of the copies of $\abe$ in the first variable (resp. in the second variable).

The horizontal composition ${\otimes}: Hom_\Ext(q_1,l_1) \otimes Hom_\Ext(q_2,l_2) \to Hom_\Ext(q_1+q_2,l_1+l_2)$  is given by:
$$\xymatrix{
{\underset{j \in \N}{\bigoplus}}s^{-j}Ext^*_{\F(\gr)}(\abe^{\otimes l_1} \otimes \Lambda^j \abe, \abe^{\otimes q_1}) \otimes {\underset{i\in \N}{\bigoplus}}s^{-i}Ext^*_{\F(\gr)}(\abe^{\otimes l_2} \otimes \Lambda^i  \abe, \abe^{\otimes q_2})
 \ar[d]^{\beta}\\
 {\underset{i, j \in \N}{\bigoplus}}s^{-j-i}Ext^*_{\F(\gr)}(\abe^{\otimes l_1} \otimes \Lambda^j \abe \otimes \abe^{\otimes l_2} \otimes \Lambda^i  \abe, \abe^{\otimes q_1+q_2}) \ar[d]^{T}\\
 {\underset{i, j \in \N}{\bigoplus}}s^{-j-i}Ext^*_{\F(\gr)}(\abe^{\otimes l_1+l_2} \otimes \Lambda^j  \abe \otimes \Lambda^i  \abe, \abe^{\otimes q_1+q_2}) \ar[d]^{c}\\
  {\underset{j+i \in \N}{\bigoplus}}s^{-j-i}Ext^*_{\F(\gr)}(\abe^{\otimes l_1+l_2} \otimes \Lambda^{j+i}  \abe, \abe^{\otimes q_1+q_2})\\
 }
$$
where $\beta$ is the exterior product, $T$ is induced by the permutation of $\Lambda^i \abe$ and $\abe^{\otimes l_2}$ and $c$ by the natural transformation $\Lambda^{i+j}  \abe \to \Lambda^{i}  \abe \otimes \Lambda^{j}  \abe$ (see (\ref{Exterior})).

The vertical composition $\circ : Hom_\Ext(q,l) \otimes Hom_\Ext(l,m) \to Hom_\Ext(q,m)$ is given by:

$$\xymatrix{
{\underset{j \in \N}{\bigoplus}}s^{-j}Ext^*_{\F(\gr)}(\abe^{\otimes l} \otimes \Lambda^j  \abe, \abe^{\otimes q}) \otimes {\underset{i \in \N}{\bigoplus}}s^{-i}Ext^*_{\F(\gr)}(\abe^{\otimes m} \otimes \Lambda^i  \abe, \abe^{\otimes l})
 \ar[d]^{\simeq}\\
 {\underset{i, j \in \N}{\bigoplus}}s^{-j}Ext^*_{\F(\gr)}(\abe^{\otimes l} \otimes \Lambda^j  \abe, \abe^{\otimes q}) \otimes s^{-i}Ext^*_{\F(\gr)}(\abe^{\otimes m} \otimes \Lambda^i  \abe, \abe^{\otimes l}) 
 \ar[d]\\
\underset{i,j \in \N}{\bigoplus}
 s^{-j} Ext^*(\abe^{\otimes l} \otimes \Lambda^j \abe, \abe^{\otimes q}) \otimes s^{-i}Ext^*(\abe^{\otimes m} \otimes \Lambda^i  \abe \otimes \Lambda^j \abe, \abe^{\otimes l} \otimes \Lambda^j  \abe) 
 \ar[d]^{\mathcal{Y}}\\
 {\underset{i,j \in \N}{\bigoplus}} s^{-i-j}Ext^*_{\F(\gr)}(\abe^{\otimes m} \otimes \Lambda^i  \abe \otimes \Lambda^j  \abe, \abe^{\otimes q})  \ar[d]^{c}\\
{\underset{i+j \in \N}{\bigoplus}}s^{-i-j}Ext^*_{\F(\gr)}(\abe^{\otimes m} \otimes \Lambda^{i+j}  \abe, \abe^{\otimes q})
 }
$$

where the second morphism is induced by the maps (\ref{E}), the third by the Yoneda product and the last one by the canonical natural transformation $\Lambda^{i+j} \abe \to \Lambda^{i} \abe\otimes \Lambda^{j}  \abe$ given in (\ref{Exterior}).
\end{defn}

\begin{rem} \label{rem-subPROP}
The PROP $\Ext^0$ (see Definition \ref{PROPExt0}) is the subPROP of $\Ext$ having the same objects and such that $\Ext^0(q,l)$ is the direct summand of $\Ext(q,l)$ for $j=0$.\end{rem}

\subsection{Calculation of $\Ext(q,l)$}
The aim of this section is to prove the following result:
\begin{prop} \label{E(q,l)}
For $q,l \in \N$ we have an isomorphism of graded $(\mathfrak{S}_q, \mathfrak{S}_l)$-bimodules 
$$\Ext(q,l)=  \underset{\substack{J \subset \textbf{q}\\ f: J\twoheadrightarrow \textbf{l}}}{\bigoplus} \ \underset{\substack{1 \leq j\leq | \textbf{q}\setminus J| \\ X_1 \amalg \ldots \amalg X_{j } =\textbf{q}\setminus J \\min(X_1)< \ldots <min(X_j)}}  {\bigoplus} \  \left(\overset{l}{\underset{i=1}{\bigotimes }}Ext^{*}_{\F(\gr; \kk)}(\abe, \abe^{\otimes |f^{-1}(i)|})\right) \otimes \left(\overset{j }{\underset{k=1}{\bigotimes }}s^{-1}Ext^{*}_{\F(\gr; \kk)}(\abe, \abe^{\otimes |X_k|})\right) $$
where $J \subset \textbf{q}$, $X_1 \amalg \ldots \amalg X_{k} $ is a partition of $\textbf{q}\setminus J$ into non-empty sets such that: $min(X_1)<\ldots <min(X_k)$.
\end{prop}
The proof of this Proposition relies on the following lemma:
\begin{lm}
For $q,l,j \in \N$ we have the following isomorphisms of graded $(\mathfrak{S}_q, \mathfrak{S}_l)$-bimodules:
\begin{equation} \label{eq-1}
Ext^*_{\F(\gr; \kk)}(\abe^{\otimes l} \otimes \Lambda^j \abe, \abe^{\otimes q}) \simeq \underset{J \subset \textbf{q}}{\bigoplus} \left(  Ext^*_{\F(\gr; \kk)}(\abe^{\otimes l} , \abe^{\otimes |J|}) \otimes Ext^*_{\F(\gr; \kk)}(\Lambda^j \abe, \abe^{\otimes|q \setminus J|})\right);
\end{equation}
\begin{equation} \label{eq-2}
Ext^*_{\F(\gr; \kk)}(\abe^{\otimes l}, \abe^{\otimes q}) \simeq \underset{  f:  \textbf{q} \twoheadrightarrow \textbf{l}}{\bigoplus} \quad \left(   \overset{l}{\underset{k=1}{\bigotimes}} Ext^*_{\F(\gr; \kk)}(\abe, \abe^{\otimes |f^{-1}(k)|}) \right)
\end{equation}
\begin{equation} \label{eq-3}
Ext^*_{\F(\gr; \kk)}(\Lambda^j\abe, \abe^{\otimes q}) \simeq  \underset{\substack{X_1 \amalg \ldots \amalg X_{j } =\textbf{q}\\min(X_1)< \ldots <min(X_j)}}  {\bigoplus} \left(  \overset{j}{\underset{k=1}{\bigotimes}} Ext^{*}_{\F(\gr)}( \abe, \abe^{\otimes |X_k|})\right).
\end{equation}
\end{lm}
\begin{proof}
For (\ref{eq-1}) we have:
\begin{eqnarray*}
Ext^*_{\F(\gr; \kk)}(\abe^{\otimes l} \otimes \Lambda^j \abe, \abe^{\otimes q}) &\simeq &Ext^*_{\F(\gr \times \gr; \kk)}(\abe^{\otimes l} \boxtimes \Lambda^j \abe, \pi_2^* (\abe^{\otimes q}))\\
& \simeq &Ext^*_{\F(\gr \times \gr; \kk)}(\abe^{\otimes l} \boxtimes \Lambda^j \abe, \underset{J \subset \textbf{q}}{\bigoplus}\abe^{\otimes |J|}  \boxtimes \abe^{\otimes |\textbf{q} \setminus J|} ) \\
&\simeq & \underset{J \subset \textbf{q}}{\bigoplus}\ Ext^*_{\F(\gr \times \gr; \kk)}(\abe^{\otimes l} \boxtimes \Lambda^j \abe,\abe^{\otimes |J|}  \boxtimes \abe^{\otimes |\textbf{q} \setminus J|} )
\end{eqnarray*}

where the first isomorphism is given by the adjunction between $\delta_2^*$ and $\pi_2^*$ (see Section \ref{Rappel-foncteurs}) and the second by (\ref{tensor2}).

Using the resolution given in Section \ref{reso-abe}, we obtain that $\abe^{\otimes l} $ and $\Lambda^j \abe$ have resolutions by finitely generated projective functors. Moreover, the values of $\abe^{\otimes n}$ and $Ext^*_{\F(\gr; \kk)}(\abe^{\otimes l} ,\abe^{\otimes |J|})$ (by Theorem \ref{Ext}) are torsion free. It follows, by the K\"unneth formula, that the graded morphism: 
 $$ Ext^*_{\F(\gr; \kk)}(\abe^{\otimes l} ,\abe^{\otimes |J|}) \otimes  Ext^*_{\F(\gr; \kk)}( \Lambda^j \abe, \abe^{\otimes |\textbf{q} \setminus J|} ) \xrightarrow{\simeq}  Ext^*_{\F(\gr \times \gr; \kk)}(\abe^{\otimes l} \boxtimes \Lambda^j \abe, \abe^{\otimes |J|}  \boxtimes \abe^{\otimes |\textbf{q} \setminus J|} ) $$
is an isomorphism.

For (\ref{eq-2}) and (\ref{eq-3}) we refer the reader to the proof of \cite[Theorem 2.3]{Vespa} and \cite[Theorem 4.2]{Vespa}.

\end{proof}

\begin{rem}
By Proposition \ref{E(q,l)} we have:
$$\Ext(n,n)\simeq  \underset{\sigma \in \mathfrak{S}_n}{\bigoplus} \ \left(\overset{n}{\underset{i=1}{\bigotimes }}Ext^*_{\F(\gr; \kk)}(\abe, \abe)\right) \simeq \kk[\mathfrak{S}_n]$$
where the last isomorphism is given by Theorem \ref{Ext}.
Hence the graded bimodule $\Ext(n,n)$ is $\kk[\mathfrak{S}_n]$ concentrated in degree $0$.
\end{rem}

\subsection{The wheeled PROP $\C_{\mathcal{P}} $}
The aim of this section is to give an isomorphism  between the PROP $\Ext$ and the wheeled PROP associated to the following wheeled operad $\mathcal{P}$.

\begin{prop} \label{wheeled-operad}
The following data define a wheeled operad denoted by $\mathcal{P}$: 
\begin{enumerate}
\item The operadic part of $\mathcal{P}$ is the operad $\mathcal{P}_0$ considered in Section \ref{operad-P0},
\item the wheeled part $\mathcal{P}_w$ is given by 
$$\mathcal{P}_w(n)=s^{-1}Ext^{*}_{\F(\gr)}(  \abe, \abe^{\otimes n });$$

\item  for $1 \leq i \leq n$, the contractions $\xi^i: \mathcal{P}_0(n) \to \mathcal{P}_w(n-1)$ are the degree $0$ maps induced by Definition \ref{contr-def}.
\end{enumerate}
 \end{prop}
 
 \begin{proof}
 We have $\mathcal{P}_w \simeq s^{-1} \mathcal{P}_0$ so $\mathcal{P}_w$ is a right $\mathcal{P}_0$-module. The contraction maps $\xi_i: \mathcal{P}_0(n) \to \mathcal{P}_w(n-1)$ are equivariant by Proposition \ref{contr-equ}.
 \end{proof}

Let $\mathcal{Q}^{\circlearrowright}$ be the wheeled completion of the operad $\mathcal{Q}$ given in Definition \ref{operade-Q}. In the following, we give an explicit description of $\mathcal{Q}^{\circlearrowright}$.

\begin{prop} \label{wheeled-completion-Q}
The wheeled operad $\mathcal{Q}^{\circlearrowright}$ is given by the following data:
\begin{enumerate}
\item \textbf{the operadic part $( \mathcal{Q}^{\circlearrowright})_0$}: \\
$ (\mathcal{Q}^{\circlearrowright})_0(n) $ is the graded vector space concentrated in degree $n-1$ and generated by $\mu_n$ defined inductively by $\mu_1=\mathrm{Id}$, $\mu_2=\mu$ and for $n\geq 2$: 
$$\mu_{n+1}=\mu \underset{1}{\circ}\mu_n$$
\item \textbf{the wheeled part  $ (\mathcal{Q}^{\circlearrowright})_w$}:\\
 $ (\mathcal{Q}^{\circlearrowright})_w(n) $ is the graded vector space concentrated in degree $n$ generated by 
 $$\xi^1(\mu) \underset{1}{\circ} \mu_n.$$
\end{enumerate}
\end{prop}

\begin{proof}

We have $( \mathcal{Q}^{\circlearrowright})_0=\mathcal{Q}$ so the description of the operadic part follows from Definition \ref{operade-Q}.

The proof of the description of $ (\mathcal{Q}^{\circlearrowright})_w(n)$ is similar to that of $(\mathcal{C}om^{\circlearrowright})_w(n)$ given in \cite[Example 5.2.5]{MMS} replacing the commutativity property by the commutativity up to signs and taking into account the fact that we have graded modules.
\end{proof}

\begin{prop} \label{wheeled-completion}
The wheeled operad $\mathcal{P}$ is isomorphic to $\mathcal{Q}^{\circlearrowright}$.
In particular, $\mathcal{P}$ is isomorphic to the wheeled completion of the quadratic operad $\mathcal{P}_0$ i.e. $\mathcal{P} \simeq \mathcal{P}_0^{\circlearrowright}$.
\end{prop}

\begin{proof}

Recall that the wheelification $(-)^{\circlearrowright}$ is the left adjoint of the forgetful functor $F$ from wheeled operads to operads. Since $\mathcal{P}$ is a wheeled operad whose the operadic part is $\mathcal{P}_0$, we have a morphism of operads $\mathcal{P}_0 \to F(\mathcal{P})$. The composition of the isomorphism of operads $f: \mathcal{Q} \to  \mathcal{P}_0$ constructed in the proof of Proposition  \ref{operade-gr}, with the morphism of operads defined above, gives a morphism of operad $ \mathcal{Q} \to  F(\mathcal{P}_0)$.
By adjunction, this morphism induces a morphism of wheeled operads
$$f^{\circlearrowright}: \mathcal{Q}^{\circlearrowright} \to \mathcal{P}.$$

By Proposition \ref{operade-gr}, the restriction of $f^{\circlearrowright}$ to the operadic parts is an isomorphism of operads given explicitly on the generators of $ (\mathcal{Q}^{\circlearrowright})_0(n) $ by 
$$f^{\circlearrowright}(\mu_n)=[\pi^{\otimes n}].$$

For the wheeled part, by Proposition  \ref{Ext}, $\mathcal{P}_w(n)=s^{-1}Ext^{*}_{\F(\gr)}(  \abe, \abe^{\otimes{n} })$ is the graded vector space concentrated in degree $n$ generated by $[\pi^{\otimes n}]$. By Proposition \ref{wheeled-completion-Q} it follows that $\mathcal{P}_w(n)$ and $(\mathcal{Q}^{\circlearrowright})_w(n)$ are isomorphic as graded vector spaces.

Since $f^{\circlearrowright}$ is a morphism of wheeled operads, the compatibility with the contractions gives:
$$f^{\circlearrowright}(\xi^1(\mu))=\xi^1(f^{\circlearrowright}(\mu))=\xi^1([\pi^{\otimes 2}])=-[\pi]$$
where the last equality is given by Definition \ref{contr-def}.
 We deduce that 
$$f^{\circlearrowright}: (\mathcal{Q}^{\circlearrowright})_w(1) \to \mathcal{P}_w(1)$$
 is an isomorphism. By Proposition \ref{wheeled-completion-Q}, the generator of $ (\mathcal{Q}^{\circlearrowright})_w(n) $  is obtained by composition of $\xi^1(\mu)$ with $\mu_n$. It follows from the compatibility of $f^{\circlearrowright}$ with the composition that, for all $n \geq 1$
$$f^{\circlearrowright}: (\mathcal{Q}^{\circlearrowright})_w(n) \to \mathcal{P}_w(n)$$
is an isomorphism.
\end{proof}

\begin{cor} \label{gen-rel}
The PROP $\C_{\mathcal{P}}$ is isomorphic to the wheeled PROP  generated by one antisymmetric operation $\mu$ of biarity $(2,1)$ in degree $1$ subject to the quadratic relation:
$$\mu(\mu \otimes 1)=-\mu(1 \otimes \mu).$$
\end{cor}

The following Theorem relates the PROP $\C_{\mathcal{P}}$ to the PROP $\Ext$ of extension groups introduced in Section \ref{PROP-E}.
\begin{thm} \label{PROP-E-explicite}
There is an isomorphim of PROPs
$$\chi: \C_{\mathcal{P}} \xrightarrow{\simeq} \Ext.$$
In particular, $\Ext$ inherits a structure of wheeled PROP via this isomorphism.
\end{thm}

\begin{proof}
Forgetting the wheeled structure on $ \C_{\mathcal{P}}$, the PROP $ \C_{\mathcal{P}}$ is generated by one antisymmetric operation $\mu$ of biarity $(2,1)$ in degree $1$ subject to the quadratic relation:
$$\mu(\mu \otimes 1)=-\mu(1 \otimes \mu)$$
and one operation $\bar{\mu}=\xi^1(\mu)$ of biarity $(1,0)$ in degree $1$.

The functor $\chi$ is defined on these generators by:
$$\chi(\mu)=[\pi^{\otimes 2}] \qquad \mathrm{and} \qquad \chi(\bar{\mu})=[\pi] \in s^{-1}Ext^*(\Lambda^1 \abe, \abe)$$
and $[\pi^{\otimes 2}] $ satisfies the quadratic relation by Proposition \ref{operade-gr}.
This defines an isomorphism of PROPs since 
$$ \C_{\mathcal{P}}(n,m) \simeq \Ext(n,m)$$
comparing the formulas given in Proposition \ref{C_P} and the one obtained in Proposition \ref{E(q,l)}.

\end{proof}

\begin{rem} \label{w-on-E}
The existence of a wheeled structure on the PROP $\Ext$ is quite surprising since it is induced by a morphism (of degree $0$):
$$Ext^*_{\F(\gr;\kk)}(\abe, \abe^{\otimes 2})  \to s^{-1} Ext^*_{\F(\gr;\kk)}(\abe, \abe).$$
By Theorem \ref{Ext}, $Ext^1_{\F(\gr;\kk)}(\abe, \abe^{\otimes 2}) \simeq \kk$ 
thus the Yoneda product with a generator in $Ext^1_{\F(\gr;\kk)}(\abe, \abe^{\otimes 2})$ gives a  morphism:
$$ Ext^*_{\F(\gr;\kk)}(\abe, \abe) \to s^1 Ext^*_{\F(\gr;\kk)}(\abe, \abe^{\otimes 2}).$$
By Theorem \ref{Ext}, this morphism is an isomorphism and the wheeled structure on the PROP $\Ext$ is given by the inverse on this morphism.

It follows that the existence of a wheeled structure  is very specific to the situation studied in this paper (i.e. $Ext$-groups  in the category $\F(\gr; \kk)$ between the tensor powers  of the functor $\abe$) and, in general, there is no such natural map.

\end{rem}

\begin{rem}
Theorem \ref{PROP-E-explicite} should be viewed as an extension, in the wheeled world, of the   isomorphism of PROPs
$$C_{\mathcal{P}_0} \  {\simeq}\  \Ext^0$$
given in \cite[Proposition 3.5]{Vespa}. More precisely there is a commutative diagram
$$\xymatrix{
C_{\mathcal{P}} \ar[r]^{\simeq} & \Ext \\
C_{\mathcal{P}_0} \ar[r]^{\simeq} \ar[u] & \Ext^0 \ar[u]
}$$
where the vertical maps are the inclusion functors.
\end{rem}
\part{The morphism of wheeled PROPs $\phi: \Ext \to \Hst$}

In this section we define a morphism from the wheeled PROP $\Ext$  to the wheeled PROP $\Hst$ of stable cohomology considered in Section \ref{PROP-H}.

\begin{thm} \label{phi}
Let $\mu$ be the generator of the wheeled PROP $\Ext$.
There is a morphism of wheeled PROPs $\phi: \Ext  \to \Hst$ given on generators by
$$\phi(\mu)=h_1.$$
\end{thm}
\begin{proof}
By the isomorphism given in Theorem \ref{PROP-E-explicite} the PROP $ \Ext$ is generated by the antisymmetric element $\mu$ of biarity $(2,1)$ in degree $1$ subject to the quadratic relation:
$$\mu(\mu \otimes 1)=-\mu(1 \otimes \mu)$$
and one operation $\bar{\mu}=\xi^1(\mu)$ of biarity $(1,0)$ in degree $1$.
The functor $\phi$ is defined on these generators by:
$$\phi_{2,1}(\mu)=h_1 \qquad \mathrm{and} \qquad \phi_{1,0}(\bar{\mu})=\overline{h_1}.$$

By Section \ref{Kawa-classes}, $h_1$ is antisymmetric and satisfies the quadratic relation. 

For $i \in \{1,2 \}$, by Proposition \ref{Prop-Kawa} and Remark \ref{xi^2}, we have $\xi^i_1(h_1)=(-1)^{i+1}\overline{h_1}$. It follows that the following diagram is commutative:
$$\xymatrix{
\Ext(2,1) \ar[r]^-{\phi_{2,1}} \ar[d]_{\xi^i_1} &\Hst(2,1) \ar[d]^{\xi^i_1}\\
\Ext(1,0) \ar[r]_-{\phi_{1,0}} & \Hst(1,0)
}$$
giving the compatibility of the wheeled PROP structures.

\end{proof}

\begin{cor}
The sub-wheeled PROP $\mathcal{K}$ of $\mathcal{H}$ is equivalent to the wheeled PROP associated to the wheeled completion of the operadic suspension of the operad $\mathcal{C}om$.
\end{cor}
\begin{proof}
 By Theorem \ref{PROP-E-explicite} and Proposition \ref{wheeled-completion}, $\Ext$ is the wheeled PROP generated by the wheeled completion of the operad $\mathcal{P}_0$ which is the suspension of the operad $\mathcal{C}om$ by Proposition \ref{operad-sus}. By Section \ref{Kawa-classes}, $\phi(\Ext) \simeq \mathcal{K}$.
\end{proof}

The morphism of wheeled PROPs $\phi: \Ext \to \Hst$ induces an explicit graded morphism on Hom-sets:
\begin{equation} \label{morphism-on-Hom}
\phi_{q,l}: \overset{q-l}{\underset{j=0}{\bigoplus}}s^{-j}Ext^*_{\F(\gr; \kk)}(\abe^{\otimes l} \otimes \Lambda^j  \abe, \abe^{\otimes q}) \to H^*_{st}(Hom_{\mathcal{V}}(\abe^{\otimes l}, \abe^{\otimes q} )).
\end{equation}

We denote by $\Ext_w$ the subPROP of $\Ext$ corresponding to forgetting the operadic part in the wheeled PROP $\Ext$ (see Remark \ref{subPROP}).
By restriction, the morphism $\phi: \Ext \to \Hst$ induces a morphism $\phi': \Ext_w \to \Hst'$ where $\Hst'$ is defined in  Definition \ref{H'}. 

One of the main and difficult result of Djament  \cite[Th\'eor\`eme 4]{Djament} gives an isomorphism of graded morphisms 
$$ \underset{j\in \N}{\bigoplus} s^{-j}Ext^*_{\F(\gr; \kk)}(\Lambda^j  \abe, \abe^{\otimes q}) \simeq H^*_{st}(\abe^{\otimes q} ).$$
By  \cite[Corollaire 3.7]{Djament} this isomorphism is induced by the morphism
$$\phi_{q,0}: \underset{j\in \N}{\bigoplus} s^{-j}Ext^*_{\F(\gr; \kk)}(\Lambda^j  \abe, \abe^{\otimes q}) \to H^*_{st}(\abe^{\otimes q} ).$$
We deduce the following result:
\begin{thm}  \cite[Th\'eor\`eme 4]{Djament} \label{Djament}
The functor $\phi': \Ext_w \to \Hst'$ is an equivalence of categories.
\end{thm}
Note that \cite[Proposition 3.5]{Djament} corresponds to the compatibility of 
the isomorphisms $\phi_{q,0}$ with the horizontal composition in the PROPs $ \Ext_w $ and $\Hst'$. Theorem \ref{Djament} gives also the compatibility of 
the isomorphisms $\phi_{q,0}$  with the action of the symmetric groups $\mathfrak{S}_q$.

For stable cohomology with coefficients given by a bivariant functor, Djament gives a conjecture in  \cite[Th\'eor\`eme 7.4]{Djament}. In particular Djament conjectures  that there exists  graded isomorphisms 
$$\overset{q-l}{\underset{j=0}{\bigoplus}}s^{-j}Ext^*_{\F(\gr; \kk)}(\abe^{\otimes l} \otimes \Lambda^j  \abe, \abe^{\otimes q}) \simeq H^*_{st}(Hom_{\mathcal{V}}(\abe^{\otimes l}, \abe^{\otimes q} )).$$

Natural candidates for maps giving these isomorphisms are the maps $\phi_{q,l}$. By functoriality these maps are compatible with horizontal and vertical compositions in the PROPs and with the contractions.

Djament's conjecture can be rephrased in the following way

\begin{conj}
The morphism $\phi$ is an isomorphism of wheeled PROPs.
\end{conj}

\bibliographystyle{amsalpha}
\bibliography{Kawazumi-Vespa.bib}
\end{document}